\pdfoutput=1
\UseRawInputEncoding
\documentclass{amsart}
\usepackage{oldgerm}
\usepackage{latexsym}
\usepackage{amsmath}
\usepackage{amscd}
\usepackage{amssymb} 
\usepackage{amsmath}
\usepackage{amsthm}
\usepackage{latexsym}
\usepackage{xcolor}

\newcommand{\nd}{\mathrm{nd}}

\newcommand{\frkd}{{\mathfrak d}}    
\newcommand{\frke}{{\mathfrak e}}    
\newcommand{\frkf}{{\mathfrak f}}

\newcommand{\cald}{{\mathcal D}}

\newcommand{\call}{{\mathcal L}}

\newcommand{\calr}{{\mathcal R}}

\newcommand{\calu}{{\mathcal U}}

\newcommand{\CC}{{\mathbb C}}

\newcommand{\QQ}{{\mathbb Q}}
\newcommand{\RR}{{\mathbb R}}

\newcommand{\ZZ}{{\mathbb Z}}

\newcommand{\smallcirc}{\lower .3em \hbox{\rm\char'27}\!}

\theoremstyle{plain}
\newtheorem{theorem}{Theorem}[section]
\newtheorem{lemma}{Lemma}[section]
\newtheorem{proposition}{Proposition}[section]

\theoremstyle{definition}

\newtheorem{conjecture}{Conjecture}[section]
\newtheorem{question}{Question}[section]
\newtheorem{remark}{{\bf Remark}}[section]
\theoremstyle{remark}

\def\smallddots{\mathinner
{\mskip1mu\raise3pt\vbox{\kern7pt\hbox{.}}
\mskip1mu\raise0pt\hbox{.}
\mskip1mu\raise-3pt\hbox{.}\mskip1mu}}

\theoremstyle{plain} 

\newtheorem{thms}{Theorem}[subsection] 
\newtheorem{lems}[thms]{Lemma}
\newtheorem{props}[thms]{Proposition}
\newtheorem{cors}[thms]{Corollary}

\theoremstyle{definition}

\newtheorem*{xrem}{Remark}

\numberwithin{equation}{section}

\theoremstyle{remark} 

\makeatletter
\@namedef{subjclassname@2020}{%
  \textup{2020} Mathematics Subject Classification}
\makeatother
\subjclass[2020]{11F46, 11F67, 11F66}

\begin{document}

\title[Rankin-Selberg convolution]{ Rankin-Selberg convolution for the Duke-Imamoglu-Ikeda lift}

\author{Hidenori Katsurada and Henry H. Kim }
\address{Hidenori Katsurada \\
Muronan Institute of Technology\\
27-1 Mizumoto Muroran 050\\
Japan}
\email{hidenori@mmm.muroran-it.ac.jp}

\address{Henry H. Kim \\
Department of mathematics \\
 University of Toronto \\
Toronto, Ontario M5S 2E4, CANADA \\
and Korea Institute for Advanced Study, Seoul, KOREA}
\email{henrykim@math.toronto.edu}
\thanks{The first author is partially supported by KAKENHI Grant Number 16H03919.
 The second author is partially supported by NSERC grant \#482564.}

\date{December 21, 2021}

\maketitle
\begin{abstract}
For two Hecke eigenforms $h_1$ and $h_2$ in the Kohnen plus space of half-integral weight, let $I_n(h_1)$ and $I_n(h_2)$ be the Duke-Imamoglu-Ikeda lift of $h_1$ and $h_2$, respectively, which are Siegel cusp forms with respect to $Sp_n(\ZZ)$. Moreover, let $E_{n/2+1/2}$ be the Cohen Eisenstein series of weight $n/2+1/2$. We then express the Rankin-Selberg convolution $R(s,I_n(h_1),I_n(h_2))$ of $I_n(h_1)$ and $I_n(h_2)$ in terms of a certain Dirichlet series $D(s,h_1,h_2,E_{n/2+1/2})$, which is similar to the triple convolution product of $h_1, h_2$ and $E_{n/2+1/2}$. We apply our formula to mass equidistribution for the Duke-Imamoglu-Ikeda lift assuming the holomorphy of $D(s,h_1,h_1,E_{n/2+1/2})$.
\end{abstract}

\section{Introduction} 
For Siegel modular forms $F_1$ and $F_2$, let $R(s,F_1,F_2)$ be the Rankin-Selberg convolution of $F_1$ and $F_2$. The first named author and Kawamura \cite{KK15}
 gave an explicit formula of $R(s,F,F)$ for a certain half-integral weight Siegel modular form $F$ related to the Duke-Imamoglu-Ikeda lift (D-I-I lift for short) in terms of well-known Dirichlet series and $L$-functions. As a result, we proved the conjecture on the period of the D-I-I lift proposed by Ikeda \cite{Ik2} (cf. Theorem \ref{th.Ikeda-conjecture}). Then a natural question arises:

\bigskip

What about $R(s,F,F)$ when $F$ is the D-I-I lift itself?

\bigskip

In this paper, when $F_1$ and $F_2$ are the D-I-I lifts, we express $R(s,F_1,F_2)$ in terms of a certain `triple convolution like Dirichlet series' attached to half integral weight modular forms.

To be more precise, for $i=1,2,3$, let $h_i$ be a Hecke eigenform in the Kohnen plus space of half-integral weight $l_i +1/2$ for $\varGamma_0(4)$. Then we define a Dirichlet series $D(s,h_1,h_2,h_3)$ (Definition \ref{D}), which can be expressed as 
 an infinite sum of Euler products of degree $10$, and it is similar to the triple convolution Dirichlet series for $h_1,h_2,h_3$. 

Let $n$ be a positive even integer, and for $i=1,2$, let $h_i$ be a Hecke eigenform in the Kohnen plus space of weight $k_i-n/2
+1/2$ with respect to $\varGamma_0(4)$, and $f_i$ the primitive form of weight $2k_i-n$ with respect to $SL_2(\ZZ)$ corresponding to $h_i$ under the Shimura correspondence. Let $I_n(h_i)$ be the D-I-I lift of $h_i$ which is a Siegel cusp form of weight $k_i$ with respect to $Sp_n(\ZZ)$. Then we express $R(s,I_n(h_1),I_n(h_2))$ in terms of $D(s,h_1,h_2,E_{n/2+1/2})$ and the tensor product $L$-function $L(s,f_1 \otimes f_2)$, where $E_{n/2+1/2}$ is the Cohen Eisenstein series of weight $n/2+1/2$ (cf. Theorem \ref{th.explicit-RS}). The method of doing it is similar to that in \cite{KK15}.
As a corollary, we prove the analytic properties (meromorphy, functional equation, residue formula) of $D(s,h_1,h_2,E_{n/2+1/2})$ (cf. Theorem \ref{th.main-result}).
Moreover, we apply our formula to mass equidistribution for the D-I-I lift assuming the holomorphy of $D(s,h_1,h_1,E_{n/2+1/2})$.

The paper is organized as follows. In Section 2, we review several $L$-functions attached to a primitive form for $SL_2(\ZZ)$, and the Rankin-Selberg convolution for a Siegel modular form. Moreover we review the D-I-I lift $I_n(h)$ of a Hecke eigenform $h$ in the Kohnen plus subspace of half-integral weight to the space of Siegel cusp forms of degree $n$, and its period relation. In Section 3, we define the Dirichlet series $D(s;h_1,h_2,h_3)$ attached to Hecke eigenforms $h_1,h_2,h_3$ in the Kohnen plus subspace, and  we  state our main results.  In Section 4, we reduce our computation to that of certain formal power series, which we call formal power series of Rankin-Selberg type. Using this, we prove our main results. In Section 5, we apply our main results to mass equidistribution for the D-I-I lift assuming the holomorphy of $D(s;h_1,h_2,E_{n/2+1/2})$.

\medskip 
\noindent
{\bf Notation.}  
Let $R$ be a commutative ring. We denote by $R^{\times}$  the unit group of $R,$  respectively.  We denote by $M_{mn}(R)$ the set of
$m \times n$-matrices with entries in $R.$ In particular put $M_n(R)=M_{nn}(R).$   Put $GL_m(R) = \{A \in M_m(R) \ | \ \det A \in R^\times \},$ where $\det
A$ denotes the determinant of a square matrix $A$. For an $m \times n$-matrix $X$ and an $m \times m$-matrix
$A$, we write $A[X] = {}^t X A X,$ where $^t X$ denotes the
transpose of $X$. Let $S_n(R)$ denote
the set of symmetric matrices of degree $n$ with entries in
$R.$ Furthermore, if $R$ is an integral domain of characteristic different from $2,$ let  ${\mathcal L}_n(R)$ denote the set of half-integral matrices of degree $n$ over $R$, that is, ${\mathcal L}_n(R)$ is the subset of symmetric
matrices of degree $n$ whose $(i,j)$-component belongs to
$R$ or $\frac{1}{2}R$ according as $i=j$ or not.  
In particular, we put ${\mathcal L}_n={\mathcal L}_n(\ZZ)$, and ${\mathcal L}_{n,p}={\mathcal L}_n(\ZZ_p)$ for a prime number $p.$ 
  For a subset $S$ of $M_n(R)$ we denote by $S^{\rm nd}$ the subset of $S$
consisting of non-degenerate matrices. If $S$ is a subset of $S_n({\RR})$ with ${\RR}$ the field of real numbers, we denote by $S_{>0}$ (resp. $S_{\ge 0}$) the subset of $S$
consisting of positive definite (resp. semi-positive definite) matrices. 
$GL_n(R)$ acts on the set $S_n(R)$ in the following way: $GL_n(R) \times S_n(R) \ni (g,A) \longmapsto {}^tg Ag \in S_n(R).$
Let $G$ be a subgroup of $GL_n(R).$ For a subset ${\mathcal B}$ of $S_n(R)$ stable under the action of $G$ we denote by ${\mathcal B}/G$ the set of equivalence classes of ${\mathcal B}$ with respect to $G.$ We sometimes identify ${\mathcal B}/G$ with a complete set of representatives of ${\mathcal B}/G.$ We abbreviate ${\mathcal B}/GL_n(R)$ as ${\mathcal B}/\sim$ if there is no fear of confusion. Two symmetric matrices $A$ and $A'$ with
entries in $R$ are said to be equivalent over $R'$ with each
other and write $A \sim_{R'} A'$ if there is
an element $X$ of $GL_n(R')$ such that $A'=A[X].$ We also write $A \sim A'$ if there is no fear of confusion. 
For square matrices $X$ and $Y$ we write $X \bot Y = \begin{pmatrix} X & O\\ O &Y \end{pmatrix}$.

For an integer $D \in \ZZ$ such that $D \equiv 0$ or $1 \ {\rm mod} \ 4,$ let ${\textfrak d}_D$ be the discriminant of $\QQ(\sqrt{D}),$ and put ${\textfrak f}_D= \sqrt{\frac{D}{{\frkd}_D}}.$ We call an integer $D$ a fundamental discriminant if it is the discriminant of some quadratic extension of $\QQ$ or $1$. 
For $d \in \QQ^{\times} \cap \ZZ$, we denote by $\Bigl(\frac{d}{*} \Bigr)$ the Dirichlet character corresponding to the extension $\QQ(\sqrt{d})/\QQ$. Here we make the convention that  $\Bigl(\frac {d}{*} \Bigr)=1$ if $d \in ({\QQ^{\times}})^2.$

We put ${\bf e}(x)=\exp(2 \pi i x)$ for $x \in {\CC}.$ For a prime number $p$ we denote by $\nu_p(*)$ the additive valuation of $\QQ_p$ normalized so that $\nu_p(p)=1,$ and by ${\bf e}_p(*)$ the continuous additive character of $\QQ_p$ such that ${\bf e}_p(x)= {\bf e}(x)$ for $x \in {\ZZ}[p^{-1}].$

\section{Preliminaries}

In this section we review $L$-functions attached to modular forms, Rankin-Selberg convolutions of Siegel modular forms, and the Duke-Imamoglu-Ikeda lift.

\subsection{Siegel modular forms}
Put $J_n=\begin{pmatrix}O_n&-1_n\\1_n&O_n\end{pmatrix}$, where $1_n$ and $O_n$ denotes the unit matrix and the zero matrix of degree $n$, respectively. 
 Furthermore, put 
$$\varGamma^{(n)}=Sp_n({\ZZ})=\{M \in GL_{2n}({\ZZ})   \ | \  J_n[M]=J_n \}.
$$
Let ${\Bbb H}_n$ be Siegel's
upper half-space of degree $n$. We define $j(\gamma,Z)=\det (CZ+D)$ for $\gamma = \begin{pmatrix} A & B \\ C & D \end{pmatrix}$ and $Z \in {\Bbb H}_n$. We note that $\varGamma^{(1)}=SL_2(\ZZ)$. Let $l$ be an integer or a half-integer. For a congruence subgroup $\varGamma$ of $\varGamma^{(n)}$, we denote by $M_{l}(\varGamma)$ the space of Siegel modular forms of weight $l$ with respect to $\varGamma$, and by $S_{l}(\varGamma)$ its subspace consisting of cusp forms.
 For two holomorphic Siegel cusp forms $F$ and $G$ of weight $l$ for  $\varGamma$, we define the Petersson product by 
$$\langle F,G \rangle=\int_{\varGamma \backslash {\Bbb H}_n} F(Z)\overline {G(Z)} (\det Y)^l d^*Z,$$
where $Y=\mathrm{Im}(Z)$ and $d^*Z$ denotes the invariant volume element on ${\Bbb H}_n$ defined by
$d^*Z=(\det Y)^{-n-1} dZ$.
 We call $\langle F, F \rangle$ the period of $F.$ 

\subsection{$L$-functions attached to modular forms}

For $f(z)=\sum_{m=1}^{\infty} c_{f}(m) {\bf e}(mz)$ be a primitive form in $S_{k}(SL_2(\ZZ))$, and for any prime number $p$, let 
$\alpha_p=\alpha_{f}(p)\in \CC^\times$ such that
\[c_{f}(p)=p^{\frac {k-1}2}(\alpha_{f}(p)+\alpha_{f}(p)^{-1}).\]
  Then  for a Dirichlet character $\chi$, we  define the Hecke $L$-function $L(s,f,\chi)$ twisted by $\chi$  as $L(s,f,\chi)=\sum_{m=1}^{\infty} c_{f}(m)\chi(m) m^{-s}$. It can be written as 
  $$L(s,f,\chi)=\prod_p \bigl \{(1-\alpha_{f}(p)\chi(p) p^{\frac {k-1}2-s})(1-\alpha_{f}(p)^{-1} \chi(p)p^{\frac {k-1}2-s})\bigr\}^{-1}.
$$
We abbreviate $L(s,f,\chi)$ as $L(s,f)$ if $\chi$ is the principal character.
Moreover, we define the adjoint $L$-function $L(s,f, \mathrm {Ad})$  as 
\[L(s,f,\mathrm {Ad})=\prod_p \{(1-\alpha_f(p)^2 p^{-s})(1-\alpha_f(p)^{-2} p^{-s})(1-p^{-s}) \}^{-1}.\]

 For primitive forms $f_i\in S_{k_i}(SL_2(\ZZ))$, $i=1,2,3$, we define the tensor product $L$-function $L(s,f_1 \otimes f_2)$ and the triple product $L$-function $L(s,f_1 \otimes f_2 \otimes f_3)$ as 
\[L(s,f_1 \otimes f_2)=\prod_p \bigl\{\prod_{a,b =\pm 1} (1-p^{\frac {k_1+k_2}2-1-s} \alpha_{f_1}(p)^a\alpha_{f_2}(p)^b\bigr\}^{-1},\]
and 
\[L(s,f_1 \otimes f_2 \otimes f_3 )=\prod_p \bigl\{\prod_{a,b,c =\pm 1} (1-p^{\frac {k_1+k_2+k_3}2-\frac 32-s} \alpha_{f_1}(p)^a\alpha_{f_2}(p)^b\alpha_{f_3}(p)^c\bigr\}^{-1}.\]
Let $k_1 \ge k_2$ and put 
\[\call(s,f_1 \otimes f_2)=(2\pi)^{-2s}\Gamma(s)\Gamma(s-k_2+1)L(s,f_1 \otimes f_2).\]
Then
 $\call(s,f_1 \otimes f_2)$ is continued holomorphically to the whole $s$-plane and has the following functional equation
\begin{equation}\label{f-f-RS}
\call(k_1+k_2-1-s,f_1 \otimes f_2)=\call(s,f_1 \otimes f_2).
\end{equation}

\subsection{Rankin-Selberg convolution of Siegel modular forms}

  For $i=1,2$ let $F_i(Z) $ be an element of $S_{l_i}(\varGamma^{(n)}).$ Then $F_i(Z)$ has the following Fourier expansion:
$$F_i(Z)= \sum_{A \in {{\mathcal L}_{m}}_{>0}} a_{F_i}(A) {\bf e}({\rm tr}(AZ)).
$$
We define the Rankin-Selberg series $R(s,F_1,F_2)$ of $F_1$ and $F_2$ as 
$$R(s,F_1,F_2)=\sum_{A \in {{\mathcal L}_{m}}_{>0}/SL_{m}({\ZZ})} \frac{a_{F_1}(A) \overline{a_{F_2}(A)}}{e(A) (\det A)^s},
$$ 
where $e(A)=\#\{X \in SL_{m}({\ZZ}) \ | \ A[X]=A \}.$ We review the analytic properties of $R(s,F_1,F_2)$ following Kalinin \cite{Kal}. 
Put 
$$
\Gamma_{{\RR}}(s)=\pi^{-s/2}\Gamma(s/2),\quad \Gamma_{\CC}(s)=2(2\pi)^{-s} \Gamma(s),\quad \xi(s)= \Gamma_{{\RR}}(s)\zeta(s).
$$
Let $E_{n,l}(Z,s)$ be the Siegel-Eisenstein series defined by
\[E_{n,l}(Z,s)=(\det Y)^s \sum_{\gamma \in \varGamma^{(n)}_{\infty} \backslash \varGamma^{(n)}} j(\gamma,Z)^{-l}|j(\gamma,Z)|^{-2s},\]
where $\varGamma^{(n)}_\infty=\{ \begin{pmatrix} A&B\\O_n&D\end{pmatrix}\in \varGamma^{(n)}\}$.

\begin{proposition}
\label{prop.integral-rep-RS}
For $i=1,2$, let $F_i \in S_{l_i}(\varGamma)$ with $l_1 \ge l_2$.
Put
\[\gamma_{n}(s)=2^{1-2sn}\pi^{-sn+\frac{n(n-1)}{4}}\prod_{i=1}^{n}\frac{\Gamma(s+\frac{1}{2}(-i+1))\Gamma(s+\frac{1}{2}(n-2l_2+2-i))} {\Gamma(s+\frac{1}{2}(n-l_1-l_2+2-j))}.\]
Then for $Re(s)\gg 0$, we
 have
\begin{align*}
&R(s,F_1,F_2)=\gamma(s)^{-1} \int_{\varGamma^{(n)} \backslash {\Bbb H}_n} F_1(Z)\overline{F_2(Z)} E_{n,l_1-l_2}(s+\tfrac {n+1}2-l_1,Z) (\det Y)^{l_2} d^*Z.
\end{align*}
In particular, if $F_1=F_2$, then
\begin{align*}
&R(s,F_1,F_1)=\gamma(s)^{-1} \int_{\varGamma^{(n)} \backslash {\Bbb H}_n} |F_1(Z)|^2 E_{n,0}(s+ \tfrac {n+1}2-l_1,Z) (\det Y)^{l_1} d^*Z.
\end{align*}
\end{proposition}
\begin{proposition}
\label{prop.fc-RS}
  Put 
$${\mathcal R}(s,F_1,F_2)=\gamma_{n}(s)\xi(2s+n+1-l_1-l_2)\prod_{i=1}^{[n/2]} \xi(4s+2n+2-2l_1-2l_2-2i)R(s,F_1,F_2).
$$
Suppose that $l_1 \ge l_2$. Then the following assertions hold: 
\begin{enumerate}
\item[{\rm (1)}] 
${\mathcal R}(s,F_1,F_2)$ has a holomorphic continuation to the whole $s$-plane with the possible exception of poles of finite order at $\frac {l_1+l_2}2-\frac j4$ for $j=0,1,...,2n+2$, and has the following functional equation:
$${\mathcal R}(l_1+l_2-(n+1)/2-s,F_1,F_2)={\mathcal R}(s,F_1,F_2).$$
\item[{\rm (2)}] Assume that $l_1=l_2=l$. Then ${\mathcal R}(s,F_1,F_2)$ is holomorphic for ${\rm Re}(s)>l,$  and has a simple pole at $s=l$ with the residue $\prod_{i=1}^{[n/2]}\xi(2i+1)\langle F_1,F_2 \rangle.$
\end{enumerate}
\end{proposition}
\begin{remark}
There is a typo in \cite{Kal}: `$\xi(4s+2n-k_1-k_2+2-2j)$' on page 195, line 5 should be `$\xi(4s+2n-2k_1-2k_2+2-2j)$'.
\end{remark}

\subsection{Review of the Duke-Imamoglu-Ikeda lift} 

For an element $a \in \QQ_p^{\times}$, we define $\chi_p(a)$ as
\[\chi_p(a)=\begin{cases} 1 & \text{ if } \QQ_p(\sqrt{a})=\QQ_p, \\
-1 & \text{ if } \QQ_p(\sqrt{a})/\QQ_p \text{ is unramified quadratic}, \\
0 & \text{ if } \QQ_p(\sqrt{a}/\QQ_p \text{ is ramified quadratic}. \end{cases}\]
 Let $T \in \call_{n,p}^\nd$ with $n$ even, let $\frkd_T$ the discriminant of $\QQ_p(\sqrt{(-1)^{n/2} \det T})/\QQ_p$, and
 $\xi_p(T)=\chi_p((-1)^{n/2} \det T).$ 
 Put
$\frke_T=( \nu_p(2^n \det T)-\nu_p(\frkd_T) )/2$.
For each $T \in {\mathcal L}_{n,p}^{\nd}$ we define the local Siegel series $b_p(T,s)$ and the primitive local Siegel series $b_p^*(T,s)$ by 
  $$b_p(T,s)=\sum_{R \in S_n(\QQ_p)/S_n({\ZZ}_p)} {\bf e}_p({\rm tr}(TR))p^{-\nu_p(\mu_p(R))s},$$  and 
  $$b_p^*(T,s)=\sum_{i=0}^{n} (-1)^i p^{i(i-1)/2} p^{(-2s+n+1)i} \sum_{D \in GL_{n}({\ZZ}_p) \backslash {\mathcal D}_{n,i}} b_p(T[D^{-1}],s), $$
  where $\mu_p(R)=[R{\ZZ}_p^n+{\ZZ}_p^n:{\ZZ}_p^n],$ and ${\mathcal D}_{n,i}=GL_n({\ZZ}_p) \begin{pmatrix}1_{n-i} & O \\ O & p1_i \end{pmatrix} GL_n({\ZZ}_p)$ for $i=0,1,\ldots,n$. 
 We remark that there exists a unique polynomial 
 $F_p(T,X)$ in $X$ such that 
 \begin{align*}
b_p(T,s)&=F_p(T,p^{-s})(1-p^{-s})\frac{\prod_{i=1}^{n/2} (1-p^{2i-2s})  }{1-\xi_p(T)p^{n/2-s}}  
\end{align*}
(cf. Kitaoka \cite{Ki1}). 
We also have
\begin{align*}
b_p^*(T,s)&=G_p(T,p^{-s})(1-p^{-s})\frac{\prod_{i=1}^{n/2} (1-p^{2i-2s})}{1-\xi_p(T)p^{n/2-s}},
\end{align*}
where $G_p(T,X)$ is a polynomial defined by
\begin{eqnarray*}
G_p(T,X)=\sum_{i=0}^{n} (-1)^i p^{i(i-1)/2} (X^2p^{n+1})^i \sum_{D \in GL_{n}({\ZZ}_p) \backslash {\mathcal D}_{n,i}} F_p(T[D^{-1}],X). 
\end{eqnarray*}
We  define a polynomial $\widetilde F_p(T,X)$ in $X$ and $X^{-1}$ as
$$\widetilde F_p(B,X)=X^{-\frke_p(T)}F_p(T,p^{-(n+1)/2}X).$$

We remark that  $\widetilde{F}_p(B,X^{-1})=\widetilde{F}_p(B,X)$ if $n$ is even (cf. \cite{Kat1}). \vspace*{2mm}
Let $T$ be an element of ${\call_n}_{>0}$ with $n$ even. Let $\frkd_T$ be the discriminant of $\QQ(\sqrt{(-1)^{n/2} \det (T)})/\QQ$. Then  
we have $(-1)^{n/2} \det (2T)/\frkd_T=\frkf_T^2$ with $\frkf_T \in \ZZ_{>0}$.
Now let $k$ be a positive even integer, and 
$\varGamma_0(4)=\Bigl\{\bigl( \begin{smallmatrix} a & b \\ c & d \end{smallmatrix} \bigr) \in Sl_2(\ZZ) \ | \ c \equiv 0 \text{ mod 4} \bigr\}$. 
Let 
 $$h(z)=\sum_{m \in {\ZZ}_{>0} \atop (-1)^{n/2}m \equiv 0, 1 \ {\rm mod} \ 4 }c_h(m){\bf e}(mz)
$$
  be a Hecke eigenform in the Kohnen plus space $S_{k-n/2+1/2}^+(\varGamma_0(4))$ and $f(z)=\sum_{m=1}^{\infty}c_f(m){\bf e}(mz)$ be 
 the primitive form in $S_{2k-n}(SL_2(\Bbb Z))$ corresponding to $h$ under the Shimura correspondence (cf. Kohnen \cite{Ko}).
We define a Fourier series $I_n(h)(Z)$ in $Z \in {\Bbb H}_n$ by
$$I_n(h)(Z)= \sum_{T \in {{\mathcal L}_n}_{> 0}} c_{I_n(h)}(T){\bf e}({\rm tr}(TZ)),\quad c_{I_n(h)}(T)=c_h(|{\textfrak d}_T|) {\textfrak f}_T^{k-n/2-1/2} \prod_p\widetilde F_p(T,\alpha_f(p)).
$$ 
Then Ikeda \cite{Ik1} showed that $I_n(h)(Z)$ is a Hecke eigenform in $S_k(\varGamma^{(n)})$ whose
standard $L$-function coincides with $\zeta(s)\prod_{i=1}^n L(s+k-i,f).$

We call $I_n(h)$ the Duke-Imamoglu-Ikeda lift (D-I-I lift for short) of $h$.

The first named author and Kawamura \cite{KK15} proved the conjecture on the period of the D-I-I lift proposed by Ikeda \cite{Ik2}.  Our result can be written as follows by using the fact that
\[L(s,f \otimes f)=L(s-2k+n+1,f, \mathrm{Ad})\zeta(s-2k+n+1).\]

\begin{theorem} \cite{KK15}
\label{th.Ikeda-conjecture}
We have 
\begin{align*}
\frac{ \langle I_n(h),\, I_n(h) \rangle}{ \langle h,\, h \rangle }=a_n2^{-2kn+4k}\pi^{-kn+k}\Gamma(k)L(k,f)\prod_{i=1}^{\tfrac n2-1} \Gamma(2k-2i)L(2k-2i,f \otimes f),
\end{align*}
with $a_n$ a non-zero constant depending only on $n$.
\end{theorem}

\section{Triple convolution product}
For $i=1,2,3$, let $h_i$
be a Hecke eigenform in the Kohnen plus space $S^+_{l_i+1/2}(\varGamma_0(4))$ of weight $l_i+1/2$ for $\varGamma_0(4),$ and 
$f_i$
be the primitive form in $S_{2l_i}(SL_2({\ZZ}))$ corresponding to $h_i$ under the Shimura correspondence. 
For a prime number $p$ and $\xi=0,{\pm 1}$, we define a polynomial $L_p(\xi;X_1,X_2,X_3,t)$ in $X_1,X_2,X_3$ and $t$ as 

\begin{eqnarray}\label{L}
&&L_p(\xi;X_1,X_2,X_3,t)=1+t\{-\xi p^{-1/2}( S_2+2)+(1+\xi^2 p^{-1})S_1-p^{-3/2}\xi\}\\ 
&&\phantom{xxxxxx} +t^2\{p^{-1}\xi^2(S_1^2-S_2-2)-\xi p^{-1/2}(S_1+S_3)-\xi p^{-3/2}S_1\} \nonumber\\
&&\phantom{xxxxxx} +t^3\{
\xi p^{-1/2}(S_1^2-S_2-2)-S_1-\xi^2p^{-1}(S_1+S_3)\}\nonumber \\
&&\phantom{xxxxxx} +t^4(-\xi^2p^{-1}(S_2+2)+\xi p^{-1/2}(1+ p^{-1})S_1-1\}+\xi t^5p^{-3/2},\nonumber
\end{eqnarray}
where $S_i=S_i(X_1,X_2,X_3)$ is the $i$-th elementary symmetric polynomial of $X_1,X_2,X_3$.
This is a polynomial in $t$ of degree at most $5.$ 
Suppose that $l_1 \ge l_2 \ge l_3$. We then define a Dirichlet series $D(s,h_1,h_2,h_3)$ as 
\begin{eqnarray}\label{D} 
&& D(s,h_1,h_2,h_3) =L(s,f_1 \otimes f_2 \otimes f_3) \sum_{d_0} c_{h_1}(|d_0|)\overline{c_{h_2}(|d_0|)}c_{h_3}(|d_0|)|d_0|^{-s} \\
&& \phantom{xxxxxxxxxxx} \times \prod_p L_p(\Bigl(\frac{d_0} p \Bigr);\widetilde c_{f_1}(p),\widetilde c_{f_2}(p),\widetilde c_{f_3}(p), p^{-2s+l_1+l_2+l_3-3/2}),\nonumber
\end{eqnarray}
where $d_0$ runs over all fundamental discriminants, and $\widetilde c_{f_i}(p)=p^{-l_i+1/2}c_{f_i}(p)$ for $i=1,2,3$.
By the estimate of the Fourier coefficients of integral and half-integral weight modular forms, this Dirichlet series is absolutely convergent if $Re(s)\gg 0$.
It is similar to the triple-convolution Dirichlet series $L(s;h_1,h_2,h_3)$ defined as 
\begin{align*}
L(s;h_1,h_2,h_3)=\sum_{m=1}^{\infty}\frac {c_{h_1}(m)\overline {c_{h_2}(m)} c_{h_3}(m)}{m^s}.
\end{align*}
Indeed, $L(s;h_1,h_2,h_3)$ can be expressed as
\begin{eqnarray*} 
&& L(s,h_1,h_2,h_3) =L(s,f_1 \otimes f_2 \otimes f_3)\times \sum_{d_0} c_{h_1}(|d_0|)\overline{c_{h_2}(|d_0|)}c_{h_3}(|d_0|)|d_0|^{-s} \\
&&\phantom{xxxxxxxxxxx} \times \prod_p M_p(\Bigl(\frac{d_0}p\Bigr);\widetilde c_{f_1}(p),\widetilde c_{f_2}(p),\widetilde c_{f_3}(p), p^{-2s+l_1+l_2+l_3-3/2}),
\end{eqnarray*}
where $M_p(\Bigl(\frac{d_0}p\Bigr);X,Y,Z, t)$ is a polynomial in $X,Y,Z$ and $t$ determined by $\Bigl(\frac{d_0} *\Bigr)$ but is not the same as 
$L_p(\Bigl(\frac{d_0} p \Bigr);X,Y,Z,t)$ in general.

For a proof, recall the following identity: For $d_0$ a fundamental discriminant, 
\begin{equation*}
c_{h_i}(m^2|d_0|)=c_{h_i}(|d_0|) \sum_{a|m} \mu(a) \Bigl(\frac{d_0} a\Bigr) a^{l_i-1} c_{f_i}(m a^{-1}),
\end{equation*}
where $\mu$ is the M\"obius function.
Now we use the fact that any integer can be written as $m^2 d_0$ for a fundamental discriminant $d_0$. Then

\begin{eqnarray*}
 L(s,h_1,h_2,h_3) =\sum_{d_0} c_{h_1}(|d_0|)\overline{c_{h_2}(|d_0|)}c_{h_3}(|d_0|)|d_0|^{-s} \sum_{m=1}^\infty A_1(m)A_2(m)A_3(m)m^{-2s},
\end{eqnarray*}
where $A_i(m)=\sum_{a|m} \mu(a) \Bigl(\frac{d_0} a\Bigr)a^{l_i-1} c_{f_i}(ma^{-1}).$ By using the fact that $A_i(m)$ is multiplicative, we have 
\begin{eqnarray*}
 L(s,h_1,h_2,h_3) =\sum_{d_0} c_{h_1}(|d_0|)\overline{c_{h_2}(|d_0|)}c_{h_3}(|d_0|)|d_0|^{-s} \prod_p \sum_{m=0}^\infty \prod_{i=1}^3 \Big(c_{f_i}(p^m)-c_{f_i}(p^{m-1}) p^{l_i-1}\Bigl(\frac{d_0} p\Bigr)\Big) p^{-2ms},
\end{eqnarray*}
where we used the convention that $c_{f_i}(p^{-1})=0$.
Use the fact that 
$$c_{f_i}(p^m)=\frac {\alpha_{f_i}(p)^{m+1}-\alpha_{f_i}(p)^{-m-1}}{\alpha_{f_i}(p)-\alpha_{f_i}(p)^{-1}} p^{m(l_i-\tfrac 12)}.
$$
Then  
\begin{eqnarray*}
&& \sum_{m=0}^\infty \Big(c_{f_i}(p^m)-c_{f_i}(p^{m-1}) p^{l_i-1}\Bigl(\frac{d_0} p\Bigr) p^{-2ms} \\
&& =\sum_{m=0}^\infty \left(\frac {\alpha_{f_i}(p)^{m+1}-\alpha_{f_i}(p)^{-m-1} -\Bigl(\frac{d_0} a
p \Bigr)p^{-\tfrac 12}(\alpha_{f_i}(p)^m-\alpha_{f_i}(p)^{-m})}{\alpha_{f_i}(p)-\alpha_{f_i}(p)^{-1}}\right) p^{m(-2s+l_i-\tfrac 12)}.
\end{eqnarray*} 

Our result follows from the following lemma, which can be easily checked.

\begin{lemma} For $i=1,2,3$ and $j=1,2$, let $\alpha_{ij}\in\Bbb C$ and $a_i=0,1$. Then 
$$\sum_{m=0}^\infty \prod_{i=1}^3 \frac {\alpha_{i1}^{m+a_i}-\alpha_{i2}^{m+a_i}}{\alpha_{i1}-\alpha_{i2}} t^m
=\frac {M_{a_1,a_2,a_3}(\alpha_{11}+\alpha_{12}, \alpha_{21}+\alpha_{22}, \alpha_{31}+\alpha_{32},t)}{\prod_{a,b,c=1,2} (1-\alpha_{1a}\alpha_{2b}\alpha_{3c} t)}.
$$
\end{lemma}

\smallskip

Let $l\geq 2$ be a positive integer. For a nonnegative integer $m$, we define the Cohen function $H(l,m)$ as 
$$H(l,m) =\begin{cases}
\zeta(1-2l), &\text{if $m=0$,} \\
L(1-l,\Bigl(\frac{(-1)^l m}*\Bigr)), & \text{if $m>0$, and  \ {\rm and} \ $(-1)^l m$ is a fundamental discriminant,}\\
0, & {\rm otherwise,}
\end{cases}
$$
where  $L(s, \Bigl(\frac{(-1)^l m}*\Bigr))$ is the Dirichlet L-function associated to $\Bigl(\frac{(-1)^l m}*\Bigr).$ We then define the Cohen Eisenstein series $E_{l+1/2}(z)$ 
 by
 $$E_{l+1/2}(z)=\sum_{m=0}^{\infty}H(l,m) {\bf e}(mz).
$$

It is known that $E_{l+1/2}(z)$ belongs to $M_{l+1/2}^+(\varGamma_0(4))$ and that the Eisenstein series $G_{2l}(z)$ of weight $2l$ with respect to $SL_2(\ZZ)$ corresponds to $E_{l+1/2}(z)$ under the Shimura correspondence. In this case, $c_{E_{l+1/2}}(|d_0|)=L(1-l,\Bigl(\frac{d_0} * \Bigr))$ and $\widetilde c_{G_{2l}}(p)=p^{l-1/2}+ p^{1/2-l}$. Therefore $D(s,h_1,h_2,E_{l+1/2})$ is expressed as 
\begin{eqnarray*}
&D(s,h_1,h_2,E_{l+1/2})=\prod_p \prod_{a,b =\pm 1} \bigl(1-p^{k_1+k_2-1-2s} \alpha_{f_1}(p)^a\alpha_{f_2}(p)^b\bigr)^{-1}\\
& \phantom{xxxxxxssssssssssssss} \times \prod_p \prod_{a,b =\pm 1} \bigl(1-p^{k_1+k_2-1-2s} \alpha_{f_1}(p)^a\alpha_{f_2}(p)^b p^{l-1}\bigr)^{-1}\\
& \phantom{xxssssssssss} \times \sum_{d_0} c_{h_1}(|d_0|)\overline{c_{h_2}(|d_0|)}L(1-l,\Bigl(\frac{d_0} * \Bigr))|d_0|^{-s} \\
& \phantom{xxxxxxxxxssssssssssssssssxxs} \times \prod_p L_p(\Bigl(\frac{d_0} p \Bigr);\widetilde c_{f_1}(p),\widetilde c_{f_2}(p),p^{l-1/2}+p^{1/2-l},  p^{-2s+l_1+l_2+l-3/2}),
\end{eqnarray*}
We note that when $l=1$, $\sum_{m=0}^{\infty}H(1,m) {\bf e}(mz)$
 is not a homomorphic modular form. However, by adding some infinite series to it, we can obtain a real analytic modular form, which will be 
denoted by $E_{3/2}(z)$. Let
$G_2(z)=\frac 1{8\pi y}-\frac 1{24}+\sum_{m=1}^\infty \sigma_1(m) {\bf e}(mz)$ be a nearly holomorphic form of weight $2$ with respect to $SL_2(\ZZ)$, where $\sigma_1(m)=\sum_{d|m} d$. 

Then $G_2$ can be regarded as the Shimura correspondence of $E_{3/2}$. 
In this case, we define $D(s,h_1,h_2,E_{3/2})$ by putting $l_3=2, c_{h_3}(|d_0|)=L(0,\Bigl(\frac{d_0} * \Bigr))$, and $\widetilde c_{f_3}=p^{-1/2}+p^{1/2}$ in (\ref{D}).

\section{Rankin-Selberg convolution of D-I-I lift}

Now our first main result can be stated as follows:
\begin{theorem}
\label{th.explicit-RS}
Let $k_1,k_2$ and $n$ be positive even integers. Given Hecke eigenforms $h_1 \in  S_{k_1-n/2+1/2}^+(\varGamma_0(4))$ and $h_2 \in S_{k_2-n/2+1/2}^+(\varGamma_0(4))$ let $ f_1 \in S_{2k_1-n}(SL_2(\ZZ))$ and $f_2 \in S_{2k_2-n}(SL_2(\ZZ))$ be the primitive forms 
corresponding to $h_1$ and $h_2,$ respectively. 
Then, we have
\begin{align}\label{explicit-formula}
&R(s,I_n(h_1),I_n(h_2))=\frac {2^{sn}}{\zeta(2s+n-k_1-k_2+1)} \Bigl( \lambda_n D(s;h_1,h_2,E_{n/2+1/2}) \prod_{i=1}^{\frac n2-1} \frac {L(2s-2i,f_1 \otimes f_2)}{\zeta(4s+2n-2k_1-2k_2+2-2i)} \\
& \phantom{xxxxxxxxxxxxxxx} + \mu_n c_{h_1}(1){\overline{c_{h_2}(1)}} \zeta(2s-k_1-k_2+n/2+1) \prod_{i=1}^{\frac{n}{2}} \frac {L(2s-2i+1,f_1 \otimes f_2)}{\zeta(4s+2n-2k_1-2k_2+2-2i)}\Bigr),\nonumber
\end{align}
where $\lambda_n$ and $\mu_n$ are non-zero rational numbers depending only on $n$.
\end{theorem}

The proof will be given in Section \ref{proof}. Now for the Dirichlet series $D(s,h_1,h_2,h_3)$, put 
\begin{align*}
&\cald(s,h_1,h_2,h_3)=2^{-s}\pi^{-2s} \frac{\Gamma(s)\Gamma(s-l_3+1/2)\Gamma(s-l_2+1/2)  \Gamma(s-l_2-l_3+1)}{\Gamma(s-l_1/2-l_2/2-[(l_3-1)/2])}\\
& \phantom{xxxxxxxxxxx} \times \xi(4s-2l_1-2l_2-2l_3+2)D(s,h_1,h_2,h_3).
\end{align*}
Then our second main result can be stated as follows. 
\begin{theorem}
\label{th.main-result} For $i=1,2$, let $h_i$ be a cuspidal Hecke eigenform in $S_{k_i-n/2+1/2}^+(\varGamma_0(4))$.
\begin{itemize}
\item [(1)] ${\mathcal D}(s,h_1,h_2,E_{n/2+1/2})$ has a meromorphic continuation to the whole $s$-plane, and has the following functional equation:
$$\cald(k_1+k_2-(n+1)/2-s;h_1,h_2,E_{n/2+1/2})=\cald(s;h_1,h_2,E_{n/2+1/2}).
$$
\item[(2)] Suppose that $k_1=k_2=k$ and $h_1=h_2$. $D(s;h_1,h_1,E_{n/2+1/2})$ has a simple pole at $s=k$ with the residue
\[d_n \frac{\langle h_1,h_1 \rangle 2^{2k}\pi^{k} L(k,f_1)}{\Gamma(k-n/2+1/2)},\]
  where  $d_n$ is a non-zero constant depending only on $n$. 
\end{itemize}
\end{theorem}
The proof will be also given in Section \ref{proof}.
\begin{remark}
We can also prove the algebraicity of $D(s;h_1,h_2,E_{n/2+1/2})$ at positive integers.
\end{remark}

\begin{remark} Special case of $n=2$. In this case, $I_2(h_i)$ is the Saito-Kurokawa lift of $h_i$. Then
$$R(s,I_2(h_1),I_2(h_2))=\frac {2^{2s-1}}{\zeta(2s+3-k_1-k_2)} D(s;h_1,h_2,E_{3/2}).
$$
As far as we know, this is a new result. From Proposition \ref{prop.fc-RS}, we see that $\mathcal D(s.h_1,h_2,E_{3/2})$ is holomorphic except possibly at $\tfrac {k_1+k_2}2-\tfrac j4$, 
$j=0,1,...,6$.
\end{remark}

But in general case, we do not know such holomorphy due to zeros of $L(s,f \otimes f)$. We can only conclude that $\mathcal D(s.h_1,h_2,E_{3/2})\prod_{i=1}^{n/2-1} \mathcal L(2s-2i,f\otimes f)$ is holomorphic except possibly at $\frac {k_1+k_2}2-\frac j4$ for $j=0,1,...,2n+2$.
So we raise the following question.
\begin{question}
Is $\mathcal D(s,h_1,h_1,E_{n/2+1/2})$ holomorphic except possibly at $\frac {k_1+k_2}2-\frac j4$ for $j=0,1,...,2n+2$?
\end{question}
We note that for $f_i\in S_{2l}(SL_2(\Bbb Z))$, $f_i(z)=\sum_{m=1}^\infty \widetilde c_{f_i}(m) m^{l-\tfrac 12}{\bf e}(mz)$,
the triple convolution product $L(s,f_1,f_2,f_3)=\sum_{m=1}^\infty \widetilde c_{f_1}(m) \widetilde c_{f_2}(m)\widetilde c_{f_3}(m) m^{-s}$ has the natural boundary $\mathrm{Re}(s)=0$ (cf. \cite[p.24]{Ku1},  \cite[p. 231]{Ku2}). Taking this into account, we raise the following question.
\begin{question}
Does the same assertion as Theorem \ref{th.main-result} hold if we replace $E_{n/2+1/2}$ by a cuspidal Hecke eigenform $h_3$ in $S^+_{l_3+1/2}(\varGamma_0(4))$?
If this is not the case, what is the natural boundary of $\mathcal D(s,h_1,h_2,h_3)$?
\end{question}

\subsection{Reduction to local computations}
In order to prove Theorem \ref{th.explicit-RS}, we reduce the problem to local computations. 

For $a,b \in \QQ_p^{\times}$ let $(a,b)_p$ the Hilbert symbol on $\QQ_p.$ Following Kitaoka \cite{Ki2}, we define the Hasse invariant $\varepsilon(A)$ of $A \in S_m(\QQ_p)^{\mathrm{nd}}$ by 
$$\varepsilon(A)=\prod_{1 \le i \le j \le m}(a_i,a_j)_p$$
if $A$ is equivalent to $a_1 \bot \cdots \bot a_m$ over $\QQ_p$ with some $a_1,a_2,...,a_m \in \QQ_p^{\times}.$ We note that this definition does not depend on the choice of $a_1,a_2,...,a_m.$ 

Now let $m$ and $l$ be positive integers such that $m \ge l.$ Then for non-degenerate symmetric matrices $A$ and  $B$ of degree $m$ and $l$ respectively with entries in ${\ZZ}_p$  we define the local density $\alpha_p(A,B)$ and the primitive local density $\beta_p(A,B)$ representing $B$ by $A$ as
$$\alpha_p(A,B)=2^{-\delta_{m,l}}\lim_{a \rightarrow
\infty}p^{a(-ml+l(l+1)/2)}\#{\mathcal A}_a(A,B),$$
 $$\beta_p(A,B)=2^{-\delta_{m,l}}\lim_{a \rightarrow
\infty}p^{a(-ml+l(l+1)/2)}\#{\mathcal B}_a(A,B),$$
where $${\mathcal A}_a(A,B)=\{X \in
M_{ml}({\ZZ}_p)/p^aM_{ml}({\ZZ}_p) \ | \ A[X]-B \in p^aS_l({\ZZ}_p)_e \},$$
$${\mathcal B}_a(A,B)=\{X \in {\mathcal A}_a(A,B) \ | \ 
  {\rm rank}_{{\ZZ}_p/p{\ZZ}_p} (X \ {\rm mod} \ p) =l \}.$$
In particular we write $\alpha_p(A)=\alpha_p(A,A).$ 
Furthermore put 
$$M(A)=\sum_{A' \in {\mathcal G}(A)} \frac{1}{e(A')}$$
for a positive definite symmetric matrix $A$ of degree $n$ with entries in ${\ZZ},$ where ${\mathcal G}(A)$ denotes the set of $SL_{n}({\ZZ})$-equivalence classes belonging to the genus of $A.$   Then
by Siegel's main theorem on the quadratic forms, we obtain 
\begin{equation}\label{kappa}
M(A)=\kappa_{n} \det A^{(n+1)/2} \prod_p \alpha_p(A)^{-1},\quad \kappa_{n}=2^{2-n}\pi^{-n(n+1)/4}\prod_{i=1}^{n}\Gamma(i/2)
\end{equation}
(cf. Theorem 6.8.1 in \cite{Ki2}). Put  
$${\mathcal F}_{p}=\{d_0 \in {\ZZ}_p \ | \ \nu_p(d_0) \le 1\}$$ 
 if $p$ is an odd prime, and  
$${\mathcal F}_{2}=\{d_0 \in {\ZZ}_2 \ | \  d_0 \equiv 1 \ {\rm mod} \ 4, \ {\rm  or} \  d_0/4  \equiv -1 \  {\rm mod} \ 4,  \ {\rm or} \ \nu_2(d_0)=3 \}.$$
From now on let $\call_{m,p}^{(0)}=S_m(\ZZ)_p)_e^{\nd}$. We note that $\call_{m,p}^{(0)}=S_m(\ZZ_p)^{\nd}$ if $p \not=2$.
For $T \in \call_{m,p}$,  
\[F^{(0)}(T,X)=F(2^{-\delta_{2,p}}T,X) \text{ and } \widetilde F^{(0)}(T,X)=\widetilde F(2^{-\delta_{2,p}}T,X),\]
where $\delta_{2,p}$ is Kronecker's delta.
We note that 
\[F^{(0)}(T,X)=F(T,X) \text{ and } \widetilde F^{(0)}(T,X)=\widetilde F(T,X) \text{ if } p\not=2.\]
A function $\omega$ on a subset ${\mathcal S}$ of $S_m(\QQ_p)$ is said to be $GL_m({\ZZ}_p)$-invariant if $\omega(A[X])=\omega(A)$ for any $A \in {\mathcal S}$ and $X \in GL_m({\ZZ}_p).$
Let $\iota_{m,p}$ be the constant function on ${\mathcal L}_{m,p}^{(0)}$ taking the value 1, and $\varepsilon_{m,p}$ the function on ${\mathcal L}_{m,p}^{(0)}$ assigning the Hasse invariant of $A$ for $A \in {\mathcal L}_{m,p}^{(0)}.$  We sometimes drop the suffix and write $\iota_{m,p}$ as $\iota_p$ or $\iota$ and the others if there is no fear of confusion.
Moreover for $d \in \QQ_p^\times \cap \ZZ_p$, let 
\[\call_{m,p}^{(0)}(d)= \{ A \in {\mathcal L}_{m,p}^{(0)} \ | \  (-1)^{[(m+1)/2]} \det A=d p^{2r} \text { with } r \in \ZZ_{\ge 0} \}.\] 
For $d_0 \in {\mathcal F}_p,l=0,1$ and a non-negative even integer $r$, put
$\kappa(d_0,r,l)=\{(-1)^{r(r+2)/8}\,((-1)^{r/2}2,\,d_0)_2\}^{l\delta_{2,p}}$.  For $d_0 \in {\mathcal F}_{p}$ and a $GL_{n}({\ZZ})_p$-invariant function $\omega_p=\varepsilon_p^l$ with $l=0,1$, we define a formal power series $H_{n,p}(d_0,\omega_p,X,Y,t) \in {\CC}[X,X^{-1},Y,Y^{-1}][[t]]$ by 
\begin{eqnarray*}
H_{n,p}(d_0,\omega_p,X,Y,t)=\kappa(d_0,n,l)^{-1}
 \sum_{A \in {\mathcal L}_{n,p}^{(0)}(d_0)/GL_{n}({\ZZ}_p)} \frac{\widetilde F_p^{(0)}(A,X)\widetilde F_p^{(0)}(A,Y)}{\alpha_p(A)} \omega_p(A) t^{\nu_p(\det A)}.
\end{eqnarray*}
We call $H_{n,p}(d_0,\omega_p,X,Y,t)$ a formal power series of Rankin-Selberg type. An explicit formula for $H_{n,p}(d_0,\omega_p,X,Y,t)$ will be given in the next section for $\omega_p=\iota_{n,p}$ and $\varepsilon_{n,p}.$ 
Let ${\mathcal F}$ denote the set of fundamental discriminants, and for $l=\pm 1,$ put 
$${\mathcal F}^{(l)}=\{ d_0 \in {\mathcal F} \  | \ ld_0 >0 \}.$$
Now for $i=1,2$ let $h_i$ be a Hecke eigenform in $S_{k_i-n/2+1/2}^+(\varGamma_0(4)),$ and $I_n(h_i)$ be as in Section 3. Let $T \in {{\mathcal L}_{n}}_{>0}.$ Then it follows from Lemma 4.1 that the $T$-th Fourier coefficient $c_{{I_n(h_i)})}(T)$  of $I_n(h_i))$ is uniquely determined by the genus to which $T$ belongs, and, by definition, it can be expressed as
$$c_{I_n(h_i)}(T)=c_{h_i}(|{\textfrak d}_T|)({\textfrak f}_T)^{k_i-n/2-1/2}\prod_p \widetilde F(T,\alpha_{i,p}),
$$
where $c_{h_i}(|{\textfrak d}_T|)$ is the $|{\textfrak d}_T|$-th Fourier coefficient of $h_i,$ and $\alpha_{i,p}$ is the 
Satake $p$-th parameter of $f_i$. Thus, by using the same method as in Proposition 2.2 of \cite{IS95}, similarly to \cite[Theorem 4.3]{KK15}, we obtain

\begin{theorem}
\label{th.local-global} 
 Let the notation and the assumption be as above. Then for ${\rm Re}(s) \gg 0,$ we have 
 \begin{align*}
 &R(s,I_n(h_1),I_n(h_2))=\kappa_{n} 2^{ns-1}
\sum_{d_0 \in {\mathcal F}^{((-1)^{n/2})} }c_{h_1}(|d_0|)\overline{c_{h_2}(|d_0|)}  |d_0|^{n/2-k_1/2-k_2/2+1/2} \\
&\phantom{xxx} \times 
\Bigl(\prod_p H_{n,p}(d_0,\iota_p,\alpha_{1,p},\alpha_{2,p},p^{-s+k_1/2+k_2/2}) + (-1)^{n(n+2)/8}\prod_p H_{n,p}(d_0,\varepsilon_p,\alpha_{1,p},\alpha_{2,p},p^{-s+k_1/2+k_2/2})\Bigr).
\end{align*}
\end{theorem}

\subsection{Formal power series associated with local Siegel series}
Throughout this section we fix a positive even integer $n.$ 
We simply write $\nu_p$ and $\chi_p$ as $\nu$ and $\chi$, respectively  if the prime number $p$ is clear from the context.
In this section we give an explicit formula of $H_{n}(d_0,\omega,X,Y,t)=H_{n,p}(d_0,\omega,X,Y,t)$  for $\omega=\iota, \varepsilon$ (cf. Theorem 5.5.1). The method is similar to that of 
 giving an explicit formula for the power series $H_{n-1,p}((d_0,\omega,X,Y,t)$ in \cite{KK15}. From now on we sometimes write $\omega=\varepsilon^l$ with $l=0$ or $1$ according as $\omega=\iota$ or $\varepsilon.$ 
 Henceforth, for a $GL_m({\ZZ}_p)$-stable subset ${\mathcal B}$ of $S_m(\QQ_p),$ we simply write  $\sum_{T \in {\mathcal B}}$ instead of $\sum_{T \in {\mathcal B}/\sim}$ if there is no fear of confusion. 
Let $m$ be an odd integer, and put
$${\mathcal L}^{(1)}_{m,p}=\{A \in {\mathcal L}_{m,p}^\nd \ | \ A \equiv -\ ^t {r}r  \ {\rm mod} \ 4{\mathcal L}_{m,p} \ { \rm for \ some } \ r \in \ZZ_p^{m} \}.
$$
For $A \in {\mathcal L}^{(1)}_{m},$ the integral vector $r \in \ZZ_p^{m}$ in the above definition is uniquely determined modulo $2\ZZ_p^{m}$ by $A,$ and is denoted by $r_A.$ Moreover it is easily shown that 
the matrix $\begin{pmatrix} 1&r_A/2\\{}^tr_A/2&{({}^tr_Ar_A+A)/4}\end{pmatrix},$ which will be denoted by $A^{(1)},$ belongs to $\call_{m+1,p},$ and that its $SL_{m+1}({\Bbb Z})$-equivalence class is uniquely determined by $A.$
We then define \[F_p^{(1)}(A,X)=F_p(A^{(1)},X), \text{ and } \widetilde F_p^{(1)}(A,X)=\widetilde F_p(A^{(1)},X).\] 

\subsubsection{Formal power series of Andrianov type}

For an $m \times m$ half-integral matrix $B$ over $\ZZ_p,$ let $(\overline{W},\overline {q})$ denote the quadratic space over $\ZZ_p/p\ZZ_p$ defined by the quadratic form $\overline {q}({\bf x})=B[{\bf x}] \ {\rm mod} \ p,$ and define the radical $R(\overline {W})$ of $\overline {W}$ by
$$R(\overline {W})=\{{\bf x} \in \overline {W} \ | \ \overline {B}({\bf x},{\bf y})=0 \ {\rm for \ any } \ {\bf y} \in \overline {W} \},$$
where $\overline {B}$ denotes the associated symmetric bilinear form of $\overline {q}.$ 
We then put $l_p(B)= {\rm rank}_{\ZZ_p/p\ZZ_p} R(\overline {W})^{\perp},$ where $R(\overline {W})^{\perp}$ is the orthogonal complement of $R(\overline {W})^{\perp}$ in $\overline {W}.$ Furthermore, in case $l_p(B)$ is even,  put $\overline {\xi}_p(B)=1$ or $-1$ according as $R(\overline {W})^{\perp}$ is hyperbolic or not. In case $l_p(B)$ is odd, we put $\overline {\xi}_p(B)=0.$ Here we make the convention that $\bar \xi_p(B)=1$ if $l_p(B)=0.$
Recall from Section 2.4, $\xi_p(B)=\chi((-1)^{n/2} det(B))$. So $\overline {\xi}_p(B)$ is different from the $\xi_p(B)$ in general, but they coincide if $B \in {\mathcal L}_{m,p} \cap \frac{1}{2}GL_m(\ZZ_p).$  For $B \in \call^{(0)}_{m,p}$, put 
$l_p^{(0)}(B)=l_p(2^{-\delta_{2,p}}B)$ and $\bar \xi_p^{(0)}(B)=\bar \xi_p(2^{-\delta_{2,p}}B)$.
 
Let $p \not=2.$ Then an element $B$ of $\call^{(0)}_{m,p}$ is equivalent, over $\ZZ_p$, to $\Theta \bot pB_1$
 with $\Theta \in GL_{m-n_1}(\ZZ_p) \cap S_{m-n_1}(\ZZ_p)$ and $B_1 \in S_{n_1}(\ZZ_p)^{\nd}.$ Then $\overline {\xi_p}^{(0)}(B)=0$ if $n_1$ is odd, and $\overline {\xi_p}^{(0)}(B)=\chi((-1)^{(m-n_1)/2} \det \Theta)$ if $n_1$ is even.
Let $p=2.$ Then an element $B \in \call^{(0)}_{m,2}$ is equivalent, over $\ZZ_2,$  to a matrix of the form $\Theta \bot 2B_1,$ where $\Theta \in GL_{m-n_1}(\ZZ_2) \cap S_{m-n_1}(\ZZ_2)_e$ and $B_1$ is one of the following two types: 
\begin{enumerate}
\item[(A.I)]
 $B_1 \in S_{n_1}(\ZZ_2)_o^{\nd}$; \vspace*{1mm}

\item[(A.II)]
 $B_1 \in  S_{n_1}(\ZZ_2)_e^{\nd}.$
 \end{enumerate}
Then $\bar \xi_p^{(0)}(B)=\chi((-1)^{(m-n_1)/2}\det \Theta)$ if $B_1$ of type (A.I) and $\bar \xi_p^{(0)}(B)=0$ if $B$ is of type (A.II).

Let $p \not=2.$ Then an element $B$ of ${\mathcal L}_{m-1,p}^{(1)}$ is equivalent, over ${\Bbb Z}_p$, to $\Theta \bot pB_1$
 with $\Theta \in GL_{m-n_1-1}({\Bbb Z}_p) \cap S_{m-n_1-1}({\Bbb Z}_p)$ and $B_1 \in S_{n_1}({\Bbb Z}_p)^{\mathrm{nd}}.$ 
Let $p=2.$ Then an element $B \in {\mathcal L}_{m-1,2}^{(1)}$ is equivalent, over ${\Bbb Z}_2,$  to a matrix of the form $2\Theta \bot B_1,$ where $\Theta \in GL_{m-n_1-2}({\Bbb Z}_2) \cap S_{m-n_1-2}({\Bbb Z}_2)_e$ and $B_1$ is one of the following three types: 
\begin{enumerate}
\item[(B.I)]
 $B_1=a  \bot 4B_2$ with $a \equiv -1 \ {\rm mod} \ 4, $ and $B_2 \in S_{n_1}({\Bbb Z}_2)_e^{\mathrm{nd}}$; \vspace*{1mm}

\item[(B.II)]
 $B_1 \in  4S_{n_1+1}({\Bbb Z}_2)^{\mathrm{nd}}$; \vspace*{1mm} 

\item[(B.III)]
 $B_1=a  \bot 4B_2$ with $a \equiv -1 \ {\rm mod} \ 4, $ and $B_2 \in S_{n_1}({\Bbb Z}_2)_o.$
 \end{enumerate}

Suppose that $p \neq 2,$ and let ${\mathcal U}={\mathcal U}_p$  be a complete set of representatives of ${\Bbb Z}_p^\times/({\Bbb Z}_p^\times)^2.$ Then, for each positive integer $l$ and $d \in {\mathcal U}_{p}$, there exists a unique, up to ${\Bbb Z}_p$-equivalence, element of $S_{l}({\Bbb Z}_p)  \cap GL_{l}({\Bbb Z}_p)$ whose determinant is $(-1)^{[(l+1)/2]}d,$ which will be denoted by  $\Theta_{l,d}.$ 
Suppose that $p=2,$ and put  ${\mathcal U}={\mathcal U}_{2}=\{1 ,5 \}.$ Then for each positive even integer $l$ and $d \in {\mathcal U}_{2}$ there exists a unique, up to ${\Bbb Z}_2$-equivalence, element of $S_{l}({\Bbb Z}_2)_{e} \cap GL_{l}({\Bbb Z}_2)$ whose determinant is  $(-1)^{l/2} d,$ which will be also denoted by  $\Theta_{l,d}.$ In particular, if $p$ is any prime number and $l$ is even, we put $\Theta_l=\Theta_{l,1}$ We make the convention that $\Theta_{l,d}$ is the empty matrix if $l=0.$ For an element $d \in {\mathcal U}$ we use the same symbol $d$ to denote the coset $d$ mod  $({\Bbb Z}_p^\times)^2.$

Let $r$ be an even positive integer. 
For $T \in \call_{r,p}^{(0)}$,  put $\frke^{(0)}(T)=\frke(2^{-1}T)$ and for $T \in \call_{r-1,p}^{(1)}$, put $\frke^{(1)}(T)=\frke(T^{(1)})$.
For $\xi=\pm 1$ and  $T \in {\mathcal L}_{r-j,p}^{(j)}$  with $j=0,1$, we define a  polynomial $\widetilde F_p^{(j)}(T,\xi,X)$ in $X$ and $X^{-1}$ by
$$\widetilde F_p^{(j)}(T,\xi,X)=X^{-\frke^{(j)}(T)}F_p^{(j)}(T,\xi p^{(-r+1)/2}X).$$
We note that $\widetilde F_p^{(j)}(T,1,X)=\widetilde F_p^{(j)}(T,X),$ but
$\widetilde F_p^{(j)}(T,-1,X)$ does not coincide with $\widetilde F_p^{(j)}(T,-X)$ in general. 
We also define a polynomial $\widetilde G_p^{(j)}(T,\xi,X,t)$ in $X,X^{-1}$ and $t$ by
$$\widetilde G_p^{(j)}(T,\xi,X,t)=\sum_{i=0}^{r-j} (-1)^i p^{i(i-1)/2}t^i \sum_{D \in GL_{r-j}(\ZZ_p) \backslash {\mathcal D}_{r-j,i}} \widetilde F^{(j)}_p(T[D^{-1}],\xi,X), $$
and put $\widetilde G_p^{(j)}(T,X,t)=\widetilde G_p^{(j)}(T,1,X,t).$
We also define a polynomial $G_p^{(j)}(T,X)$ in $X$ by
\begin{eqnarray*}
G_p^{(j)}(T,X)=\sum_{i=0}^{-jr} (-1)^i p^{i(i-1)/2} (X^2p^{r+1-j})^i \sum_{D \in GL_{r-j}(\ZZ_p) \backslash {\mathcal D}_{r-j,i}} F_p^{(j)}(T[D^{-1}],X). 
\end{eqnarray*}
We note that
$$\widetilde G_p^{(j)}(T,X,1)= X^{-\frke^{(j)}(T)} G_p^{(j)}(T,Xp^{-(r+1)/2}).$$
\begin{xrem} There are typos in \cite{KK15}:

Page 459, line 12: For `$\widetilde F_p^{(j)}(T,\xi X)$', read `$\widetilde F_p^{(j)}(T,\xi p^{(-r+1)/2}X)$'.

Page 459: line 19: For  `$\widetilde G_p^{(j)}(T,p^{-(m+1)/2}X)$', read `$\widetilde G_p^{(j)}(T,\xi p^{(-r+1)/2}X)$'
\end{xrem}
Suppose that $p \neq 2,$ and let ${\mathcal U}={\mathcal U}_p$  be a complete set of representatives of $\ZZ_p^\times/(\ZZ_p^\times)^2.$ Then, for each positive integer $l$ and $d \in {\mathcal U}_{p}$, there exists a unique, up to $\ZZ_p$-equivalence, element of $S_{l}(\ZZ_p)  \cap GL_{l}(\ZZ_p)$ whose determinant is $(-1)^{[(l+1)/2]}d,$ which will be denoted by  $\Theta_{l,d}.$ 
Suppose that $p=2,$ and put  ${\mathcal U}={\mathcal U}_{2}=\{1 ,5 \}.$ Then for each positive even integer $l$ and $d \in {\mathcal U}_{2}$ there exists a unique, up to $\ZZ_2$-equivalence, element of $S_{l}(\ZZ_2)_{e} \cap GL_{l}(\ZZ_2)$ whose determinant is  $(-1)^{l/2} d,$ which will be also denoted by  $\Theta_{l,d}.$ In particular, if $p$ is any prime number and $l$ is even, we put $\Theta_l=\Theta_{l,1}$ We make the convention that $\Theta_{l,d}$ is the empty matrix if $l=0.$ For an element $d \in {\mathcal U}$ we use the same symbol $d$ to denote the coset $d$ mod  $(\ZZ_p^\times)^2.$ Then by definition, we have the following lemma.

\begin{lems}
\label{lem.induction-formula-Siegel-series} 
Let $m$ be a positive even integer.  Let $B \in {\mathcal L}_{m,p}^{(0)}.$ Then 
$$\widetilde F_p^{(0)}(B,X)=\sum_{B' \in  {\mathcal L}_{m,p}^{(0)}/ GL_{m}(\ZZ_p)  } X^{-\frke^{(0)}(B')}\frac{\alpha_p(B',B)}{\alpha_p(B')} \times  G_p^{(0)}(B',p^{(-m-1)/2}X)(p^{-1}X)^{(\nu(\det B)-\nu(\det B'))/2}.$$
\end{lems}

\begin{lems} 
 \label{lem.primitive-Siegel-series}
 Let $n$ be a positive even integer. Let $B \in {\mathcal L}^{(0)}_{n,p}$. Throughout (1) and (2), for  $\Theta \in GL_{n-n_1}(\ZZ_p)$ with $n_1$ even, put  $\xi=\chi((-1)^{(n-n_1)/2} \det \Theta)$. Here we make the convention that $\xi=1$ if $n=n_1$.
\begin{itemize}
\item[(1)] Let $p \not=2, $ and suppose that  $B=\Theta \bot pB_1$  with $\Theta \in GL_{n-n_1}(\ZZ_p) \cap S_{n-n_1}(\ZZ_p)$ and $B_1 \in S_{n_1}({\Bbb Z}_p)^{\mathrm{nd}}.$ 
Then 
\begin{eqnarray*}
\lefteqn{
G_p^{(0)}(B,Y)
}{ } \\
&=& \begin{cases}
1 & \text{ if } n_1=0 \\
\displaystyle (1-\xi_p(B) p^{n/2}Y)\prod_{i=1}^{n_1/2-1}(1-p^{2i+n}Y^2)(1+p^{n_1/2+n/2}\xi Y)
& \text{ if $n_1$ is positive and even}  \\
\displaystyle (1-\xi_p(B) p^{n/2}Y)\prod_{i=1}^{(n_1-1)/2}(1-p^{2i+n}Y^2)
& \text{ if $n_1$ is odd}
\end{cases}.
\end{eqnarray*}
\item[(2)] Let $p=2.$ Suppose that $n_1$ is even and that $B=\Theta \bot 2B_1 $  with  $\Theta \in GL_{n-n_1}(\ZZ_2) \cap S_{n-n_1}(\ZZ_2)_e$ and $B_1 \in S_{n_1}({\Bbb Z}_2)^{\nd}.$  Then 
\begin{eqnarray*}
\lefteqn{G_2^{(0)}(B,Y)}{} \\
&=& \begin{cases}
1 & \text{ if } n_1=0 \\
{\displaystyle (1-\xi_2(B) 2^{n/2}Y)\prod_{i=1}^{n_1/2-1}(1-2^{2i+n}Y^2)(1+2^{n_1/2+n/2}\xi Y)}
& \text{ if }  n_1 \text{ is positive and }  B_1 \in S_{n_1}(\ZZ_2)_e,\\
(1-\xi_2(B) p^{n/2}Y)\displaystyle {\prod_{i=1}^{n_1/2}(1-2^{2i+n}Y^2)}
& \text { if } B_1 \in S_{n_1}(\ZZ_2)_o . 
  \end{cases}.
\end{eqnarray*}
\end{itemize}
\end{lems}

\begin{proof} The assertion follows from Lemma 9 of \cite{Ki1}. 
\end{proof}  


For $A \in \call_{m,p}^{(0)}$, we define Andrianov's polynomial $B_p^{(0)}(v;A)$ as follows:  
$$B_p^{(0)}(v,A)=
\left\{\begin{array}{ll} (1+v)(1-\bar \xi_p^{(0)}(A)p^{-l/2}v)
\prod_{i=1}^{l/2-1}(1-p^{-2i}v^2) & \ {\rm if} \ l \ {\rm is \ even} \\ (1+v)
\prod_{i=1}^{(l-1)/2}(1-p^{-2i}v^2) & \ {\rm if} \ l \ {\rm is \ odd}
\end{array}\right.
$$ 
with $l=l^{(0)}_p(A)$. Here we understand that we have $B_p^{(0)}(v,A)=1$ if $l=0.$
Then by definition we have the following:

 
\begin{lems}
\label{lem.Andrianov-B}
Let $n$ be the fixed positive even integer. Let $B \in {\mathcal L}_{n,p}^{(0)}.$  Throughout (1) and (2), for  $\Theta \in GL_{n-n_1}(\ZZ_p)$ with $n_1$ even, put  $\xi=\chi((-1)^{(n-n_1)/2} \det \Theta)$. Here we make the convention that $\xi=1$ if $n_1=n$. \vspace*{1mm}
\begin{itemize}
\item[(1)] Let $p \not=2, $ and suppose that  $B=\Theta \bot pB_1$  with $d \in {\mathcal U}$ and $B_1 \in S_{n_1}(\ZZ_p)^{\mathrm{nd}}.$ 
Then 
\[
B_p^{(0)}(B,t)=\begin{cases}
1 & \text{ if } n_1=n \\
\displaystyle (1+t)  (1-\xi p^{(n_1-n)/2}t)\prod_{i=1}^{(n-n_1-2)/2}(1-p^{-2i}t^2), &  
\text{ if }  n_1 \text{is even and} n_1<n, \\
\displaystyle (1+t)\prod_{i=1}^{(n-n_1-1)/2}(1-p^{-2i}t^2), & 
\text{ if }   n_1 \text{ is odd}. 
\end{cases}.
\]
\item[(2)] Let $p=2,$ and suppose that $B=\Theta \bot 2B_1$  with $\Theta \in S_{n-n_1}(\ZZ_2)_e \cap  GL_{n-n_1}(\ZZ_2) $ and $B_1 \in S_{n_1}(\ZZ_2)^{\nd}.$  Then 
\begin{eqnarray*}
B_2^{(0)}(B,t)=\begin{cases}
1 & \text{ if } n_1=n\\
\displaystyle (1+t)(1-\xi 2^{(n_1-n)/2}t)\prod_{i=1}^{(n-n_1-2)/2}(1-2^{-2i}t^2), &  
\text{ if } n_1<n \text{ and }  B_1 \in S_{n_1}(\ZZ_2)_e , \\
\displaystyle (1+t)\prod_{i=1}^{(n-n_1-2)/2}(1-2^{-2i}t^2), & 
\text{ if }  B_1 \in S_{n_1}(\ZZ_2)_o.
\end{cases}
\end{eqnarray*}
\end{itemize}
\end{lems}

Let $m$ be a positive even integer. For an element   $T \in \call_{m,p}^{(0)}$.  put
$$R(T,X,t)=\sum_{W \in M_{m}(\ZZ_p)^{\nd}/GL_{m}(\ZZ_p)}  \widetilde F_p^{(0)}(T[W],X)t^{\nu(\det W)}.$$
This type of formal power series  was first introduced by Andrianov \cite{A} to study the standard $L$-function of Siegel modular form of integral weight. Therefore we call it the formal power series of Andrianov type. (See also B{\"o}cherer \cite{Bo}.)
The following proposition is due to \cite [Proposition 5.2]{KK15}.
   
\begin{props}
\label{prop.Andrianov-series} 
Let $m$ be a positive even integer.  Let $T \in \call_{m,p}^{(0)}.$ Then  
 $$\sum_{B \in {\mathcal L}_{m,p}^{(0)}}\frac{\widetilde F_p^{(0)}(B,X)\alpha_p(T,B)}{\alpha_p(B)}t^{\nu(\det B)}= t^{\nu(\det T)}R(T,X,p^{-m}t^2).$$
\end{props}
   
The following theorem is due to \cite{A}.  

\begin{thms}
\label{th.Andrianov-identity}  
Let $T$ be an element of ${\mathcal L}_{n,p}^{(0)}.$ Then 
$$R(T,X,t)=\frac{B_p^{(0)}(T,p^{(n-1)/2}t)\widetilde G_p^{(0)}(T,X,t)} {\prod_{j=1}^{n}(1-p^{j-1}X^{-1}t)(1-p^{j-1}Xt)}.$$
\end{thms}

For a variable $Y$ we introduce the symbol $Y^{1/2}$ so that $(Y^{1/2})^2=Y,$ and for an integer $a$, write $Y^{a/2}=(Y^{1/2})^a.$ 
For $\omega=\varepsilon^l$ define a formal power series $\widetilde R_{n}(d_0,\omega,X,Y,t)$ in $t$ by 
\begin{eqnarray*}
&& \widetilde R_{n}(d_0,\omega,X,Y,t)=\kappa(d_0,n,l)^{-1}Y^{\nu(d_0)/2}\sum_{B' \in {\mathcal L}_{n,p}^{(0)}(d_0)}\hspace*{-2.5mm}\frac{ \widetilde G_p^{(0)}(B',X,p^{-n-1}Yt^2)}{\alpha_p(B') } \\
&& \phantom{xxxxxxxxxxxx} \times(tY^{-1/2})^{\nu(\det B')} B_p^{(0)}(B',p^{-(n+3)/2}Yt^2) G_p^{(0)}(B',p^{-(n+1)/2}Y)\omega(B').
\end{eqnarray*}
This is an element of ${\CC}[X,X^{-1},Y^{1/2},Y^{-1/2}][[t]].$

 \begin{thms}
\label{th.local-RS-series} 
 We have 
 $$H_{n}(d_0,\omega,X,Y,t)=\frac{\widetilde R_{n}(d_0,\omega,X,Y,t)}{\prod_{j=1}^n(1-p^{j-1-n}XYt^2)(1-p^{j-1-n}X^{-1}Yt^2)}$$
 for $\omega=\varepsilon^l.$
\end{thms}

 \begin{proof}  By Lemma \ref{lem.induction-formula-Siegel-series}, we have
 \begin{eqnarray*}
\lefteqn{ H_{n}(d_0,\omega,X,Y,t)=\sum_{B \in {\mathcal L}_{n,p}^{(0)}(d_0)}\frac{\widetilde F_p^{(0)}(B,X)}{\alpha_p(B)} \omega(B)t^{\nu(\det B)} } \\
&&\phantom{xxxxxxxxx} \times \sum_{B' \in {\mathcal L}_{n,p}^{(0)}} \frac{Y^{-\frke^{(0)}(B')}G_p^{(0)}(B',p^{-(n+1)/2}Y) \alpha_p(B',B)}{\alpha_p(B')} (p^{-1}Y)^{(\nu(\det B)-\nu(\det B'))/2}.
\end{eqnarray*}
Let $B$ and $B'$ be elements of ${\mathcal L}_{n,p}^{(0)},$ and suppose that $\alpha_p(B',B) \not=0.$ Then we note that $B \in {\mathcal L}_{n,p}^{(0)}(d_0)$ if and only if $B' \in {\mathcal L}_{n,p}^{(0)}(d_0).$ Hence by Proposition \ref{prop.Andrianov-series} and Theorem \ref{th.Andrianov-identity} we have 
 \begin{eqnarray*}
&&H_{n}(d_0,\omega,X,Y,t) \\
&&= \sum_{B' \in {\mathcal L}_{n,p}^{(0)}(d_0)} \tfrac {G_p^{(0)}(B',p^{-(n+1)/2}Y) Y^{-\frke^{(0)}(B')}}{\alpha_p(B')} (pY^{-1})^{\nu(\det B')/2}\omega(B') \times \sum_{B \in {\mathcal L}_{n,p}^{(0)}} \tfrac {\widetilde F_p^{(0)}(B,X) \alpha_p(B',B)}{\alpha_p(B)} (t^2p^{-1}Y)^{\tfrac {\nu(\det B)}2} \\
&&=\sum_{B' \in {\mathcal L}_{n,p}^{(0)}(d_0)} \frac{ G_p^{(0)}(B',p^{-(n+1)/2}Y)Y^{-\frke^{(0)}(B')} }{\alpha_p(B') } t^{\nu(\det B')}\omega(B')R(B',X,t^2Yp^{-n-1}) \\
&&=\sum_{B' \in {\mathcal L}_{n,p}^{(0)}(d_0)} \tfrac {\widetilde G_p^{(0)}(B',X,p^{-n-1}Yt^2)}{\alpha_p(B')} \omega(B') Y^{\nu(d_0)/2} (tY^{-1/2})^{\nu(\det B')} \times \tfrac {B_p^{(0)}(B',p^{-(n+3)/2}Yt^2) G_p^{(0)}(B',p^{-(n+1)/2}Y)}{\prod_{j=1}^n(1-p^{j-2-n}XYt^2)(1-p^{j-2-n}X^{-1}Yt^2)}\\
&&= \frac{\widetilde R_{n}(d_0,\omega,X,Y,t)}{ \prod_{j=1}^n(1-p^{j-2-n}XYt^2)(1-p^{j-2-n}X^{-1}Yt^2)}
\end{eqnarray*} 

\end{proof}

\subsubsection{Formal power series of Koecher-Maass type and of modified Koecher-Maass type} 

Let $r$ be an  even positive integer. For $d_0 \in {\mathcal F}_p$ and $l=0,1$, let $\kappa(d_0,r,l)$ be the rational  number defined in 
Section 4.1. We also define $\kappa(d_0,r-1,l)$ as
\begin{align*}
 \kappa(d_0,r-1,l)&= \{(-1)^{lr(r-2)/8} 2^{-(r-2)(r-1)/2}\}^{\delta_{2,p}} \times ((-1)^{r/2},(-1)^{r/2}d_0 )_p^l\,\, p^{-(r/2-1)l\nu(d_0)}. 
\end{align*}
We  define a  formal power series $P_{r-j}^{(j)}(d_0,\omega,\xi,X,t)$ in $t$ by 
\begin{align*}
& P_{r-j}^{(j)}(d_0,\omega,\xi,X,t)=\kappa(d_0,r-j,l)^{-1} t^{(-r+j+1)\delta_{2,p}j}\times \sum_{B \in {\mathcal L}_{r,p}^{(j)}(d_0)}  \hspace*{-2.5mm}\frac{\widetilde F_p^{(j)}(B,\xi,X)}{\alpha_p(B)}\omega(B)t^{\nu(\det B)}
\end{align*}
for $\omega=\varepsilon^l$ with $l=0,1.$ 
In particular we put $P_{r-j}^{(j)}(d_0,\omega,X,t)=P_{r-j}^{(j)}(d_0,\omega,1,X,t).$
 This type of formal power series appears in an explicit formula of the Koecher-Maass series associated with 
 the Siegel Eisenstein series and the Duke-Imamo{\=g}lu-Ikeda lift. Therefore we say that this formal power series is of Koecher-Maass type  (cf. \cite{KK15}). 
Moreover for $d_0,r,j,\xi$ above and  a positive integer $m$, we also define a formal power series $\widetilde P_{r}(m;d_0,\omega,\xi,X,Y,t)$  in $t$ by 
\begin{align*}
&\widetilde P_{r-j}^{(j)}(m;d_0,\omega,\xi,X,Y,t)=\kappa(d_0,r-j,l)^{-1}Y^{\nu(d_0)/2}(tY^{-1/2})^{(-r+j+1)\delta_{2,p}j}\\
&\phantom{xxxxxxxxxxxsssssxxx} \times \sum_{B' \in {\mathcal L}_{r,p}^{(j)}(d_0)}\hspace*{-2.5mm} \frac{\widetilde G_p^{(j)}(B',\xi,X,p^{-m}t^2Y)}{\alpha_p(B') } \omega(B')(tY^{-1/2})^{\nu(\det (B'))}
\end{align*}
for $\omega=\varepsilon^l.$ Here we make the convention that $\widetilde P_0(n;d_0,\omega,\xi,X,Y,t)=1$ or $0$ according as $\nu(d_0)=0$ or not.
  The relation between $\widetilde P_{r-j}^{(j)}(m;d_0,\omega,\xi,X,Y,t)$ and 
$P_{r-j}^{(j)}(d_0,\omega,\xi,X,t)$ will be given in the following proposition (cf. \cite{KK15}, Proposition 5.5):

\begin{props}
\label{prop.local-modified-KM} 
Let $r$ be a positive even integer. Let $\omega=\varepsilon^l$ with $l=0,1$, and $j=0,1$.  Then 
\begin{align*}
&\widetilde P_{r-j}^{(j)}(m;d_0,\omega,\xi,X,Y,t)=Y^{\nu(d_0)/2}P_{r-j}^{(j)}(d_0,\omega,\xi,X,tY^{-1/2}) \prod_{i=1}^{r-j} (1-t^4p^{-m-r+j-2+i}).
\end{align*}
\end{props}
We also recall an explicit formula for $P_{r-j}^{(j)}(d_0,\iota,\xi,X,t)$ (cf. \cite{KK15}, Corollary 5.7).
\begin{thms}
\label{th.explicit-local-KM}
Let $d_0 \in {\mathcal F}_{p}$ and $\xi_0=\chi(d_0).$ Let $\xi=\pm 1.$ Let $m$ be even. Put $\phi_r(x)=\prod_{i=1}^r (1-x^i)$ for a positive integer $r.$ 
Then
\begin{itemize}
\item[(1)] 
  \begin{align*}
&P_m^{(0)}(d_0,\iota,\xi,X,t)=\frac{(p^{-1}t)^{\nu(d_0)}}{\phi_{m/2-1}(p^{-2})(1-p^{-m/2}\xi_0)}\\
 & \phantom{xxxxxxxxsssx} \times  \frac{(1+t^2p^{-m/2-3/2}\xi)}{ 
  (1-p^{-2}Xt^2)(1-p^{-2}X^{-1}t^2)\prod_{i=1}^{m/2} (1-t^2p^{-2i-1}X)(1-t^2p^{-2i-1}X^{-1})  }\\
& \phantom{xxxxxxxxxxss} \times (1+t^2p^{-m/2-5/2}\xi\xi_0^2)-
  \xi_0 t^2p^{-m/2-2}(X+X^{-1}+p^{1/2-m/2}\xi +p^{-1/2+m/2}\xi)
\end{align*}
  \begin{align*}P_m^{(0)}(d_0,\varepsilon,\xi,X,t)=\frac{1}{\phi_{m/2-1}(p^{-2})(1-p^{-m/2}\xi_0)}\frac{\xi_0^2}{ 
  \prod_{i=1}^{m/2} (1-t^2p^{-2i}X)(1-t^2p^{-2i}X^{-1})  }. \end{align*}
   \noindent
\item[(2)] 
 \begin{align*} 
&P_{m-1}^{(1)}(d_0, \iota,\xi,X,t) \\
&=\frac{(p^{-1}t)^{\nu(d_0)} (1-\xi_0 t^2 p^{-5/2}\xi)}{ 
  (1-t^2p^{-2}X)(1-t^2p^{-2}X^{-1})\prod_{i=1}^{(m-2)/2} (1-t^2p^{-2i-1}X)(1-t^2p^{-2i-1}X^{-1}) \phi_{(m-2)/2}(p^{-2})},\\
&P_{m-1}^{(1)}(d_0,\varepsilon,\xi,X,t)=\frac {(p^{-1}t)^{\nu(d_0)} (1-\xi_0 t^2 p^{-1/2-r} \xi)}{ 
  \prod_{i=1}^{m/2} (1-t^2p^{-2i}X)(1-t^2p^{-2i}X^{-1}) \phi_{(m-2)/2}(p^{-2})}.
\end{align*}
\end{itemize}
\end{thms}  
 
\bigskip
Now let $r$ be an even integer. Then  we define  a partial series $ Q_{r-j}^{(j)}(m;d_0,\omega,\xi,X,Y,t)$ of $\widetilde P_{r-j}^{(j)}(m;d_0,\omega,\xi,X,Y,t)$ as follows: First let $p\not=2.$ Then put 
\begin{align*}
&Q_{r}^{(0)}(m;d_0,\varepsilon^l,\xi,X,Y,t)=Y^{\nu(d_0)/2} \\
 & \phantom{xxxxxx} \times \sum_{B' \in S_{r}(\ZZ_p,d_0) \cap S_{r}(\ZZ_p)} \frac{\widetilde G_p^{(0)}(pB',\xi,X,p^{-m}t^2Y)}{ \alpha_p(pB') } \varepsilon(pB')^l(tY^{-1/2})^{\nu(\det pB')},
 \end{align*}
\begin{align*}
&Q_{r-1}^{(1)}(m;d_0,\varepsilon^l,\xi,X,Y,t)=\kappa(d_0,r-1,l)^{-1}Y^{\nu(d_0)/2} \\
 &\phantom{xxxxxx} \times \sum_{B' \in p^{-1}S_{r-1}(\ZZ_p,d_0) \cap S_{r-1}(\ZZ_p)} \frac{\widetilde G_p^{(1)}(pB',\xi,X,p^{-m}t^2Y)}{\alpha_p(pB') } \varepsilon(pB')^l(tY^{-1/2})^{\nu(\det pB')}.
 \end{align*}
Next let $p=2.$  Then put
\begin{align*}
& Q_{r-1}^{(1)}(m;d_0,\varepsilon^l,\xi,X,Y,t)=\kappa(d_0,r,l)^{-1}(tY^{-1/2})^{2-r}Y^{\nu(d_0)/2} \\
& \phantom{xxxxxx} \times \sum_{B' \in S_{r-1}(\ZZ_2,d_0) \cap S_{r-1}(\ZZ_2)} \frac{\widetilde G_2^{(1)}(4B',\xi,X,2^{-m}t^2Y)}{ \alpha_2(4B') } \varepsilon(4B')^l(tY^{-1/2})^{\nu(\det (4B'))},
\end{align*} 
\begin{align*}
&Q_{r}^{(0)}(m;d_0,\varepsilon^l,\xi,X,Y,t)=\kappa(d_0,r,l)^{-1}Y^{\nu(d_0)/2} \\ 
&\phantom{xxxxxx} \times \sum_{B' \in  S_{r}(\ZZ_2,d_0) \cap S_r(\ZZ_2)_e} \frac{\widetilde G_2^{(0)}(2B',\xi,X,2^{-m}t^2Y)}{\alpha_2(2B') } \varepsilon(2B')^l(tY^{-1/2})^{\nu(\det (2B'))}. 
\end{align*}
Here we make the convention that $Q_{0}^{(0)}(n;d_0,\varepsilon^l,\xi,X,Y,t)=1$  or $0$ according as $\nu(d_0)=0$ or not. 
 To consider the relation between 
$\widetilde P_{r-j}^{(j)}(m;d_0,\varepsilon^l,\xi,X,Y,t) \ $ and $\ Q_{r-j}^{(j)}(m;d_0,\varepsilon^l,\xi,X,Y,t),$ 
and to express $\widetilde R_{n}(d_0,\varepsilon^l,X,Y,t)$ in terms of 
$\widetilde P_{r-j}^{(j)}(m;d_0,\varepsilon^l,\xi,X,Y,t),$ we provide some more preliminary results. 
Henceforth, for a while, we abbreviate $S_r(\ZZ_p)$ and $S_r(\ZZ_p,d)$ as $S_{r,p}$ and $S_{r,p}(d),$ respectively.
Furthermore we abbreviate $S_r(\ZZ_2)_x$ and $S_r(\ZZ_2,d)_x$ as $S_{r,2;x}$ and $S_{r,2}(d)_x,$ respectively, for $x=e,o.$  

 Let $\widetilde R_{n}(d_0,\omega,X,Y,t)$ be the formal power series defined at the beginning of Section 5. We express $\widetilde R_{n}(d_0,\omega,X,Y,t)$ in terms of  $Q_{2r}^{(0)}(n;d_0d, \omega, \chi(d), X,Y,t)$ and $Q_{2r+1}^{(1)}(n;d_0,\omega,1, X,Y,t).$ Henceforth, for $d_0 \in {\mathcal F}_p$ and non-negative integers $m,r$ such that $r \le m,$ put
 ${\mathcal U}(m,r,d_0)=\{1\},{\mathcal U} \cap \{d_0\},$ or ${\mathcal U}$ according as $r=0,r=m,$ or $1 \le  r \le m-1.$ 
\begin{thms}    
\label{th.formula-for-R} 
Let $d_0 \in {\mathcal F}_p$, and $\xi_0=\chi(d_0).$  For $d \in {\mathcal U}(n,n-2r,d_0)$ put
$$D_{2r}(d,Y,t)= (1+p^{r-1/2}\chi(d) Y)(1-p^{-n-3/2+r}\chi(d)Yt^2) (1+p^{-n/2+r}\chi(d)).$$
\noindent
\begin{enumerate}
\item[{\rm (1)}] Let $\omega=\iota,$ or  $\nu(d_0)=0.$ Then  
\begin{align*}
&\widetilde R_{n}(d_0,\omega,X,Y,t) =  \sum_{r=0}^{n/2} \frac{\prod_{i=m_0}^{r-1} (1-p^{2i-1}Y^2) \prod_{i=1}^{(n-2r)/2}(1-p^{-2i-n-1}Y^2t^4)}{(1+p^{-1/2}\xi_0Y)(1 -p^{-(n+3)/2} Yt^2)\phi_{(n-2r)/2}(p^{-2})}\\
&\phantom{xxxxxxxxxxxxx} \times \sum_{d \in {\mathcal U}(n,n-2r,d_0)} \frac{D_{2r}(d,Y,t)}{2^{1-\delta_{0,r}}}\widetilde Q_{2r}^{(0)}(n;d_0d, \omega, \chi(d), X,Y,t)\\
&\phantom{xxxx} + \sum_{r=1}^{(n-2)/2} \frac{\prod_{i=m_0}^{r} (1-p^{2i-1}Y^2) \prod_{i=1}^{(n-2r)/2}(1-p^{-2i-n-1}Y^2t^4)}{(1+p^{-1/2}\xi_0Y)(1 -p^{-(n+3)/2} Yt^2)\phi_{(n-2r-2)/2}(p^{-2})} \widetilde Q_{2r+1}^{(1)}(n;d_0,\omega,1, X,Y,t).
\end{align*}
where $m_0=1$ or $0$ according as $\xi_0=0$ or not.
\item[{\rm (2)}] Let $\nu(d_0) > 0.$  Then  
\begin{align*}
& \widetilde R_{n}(d_0,\varepsilon,X,Y,t)=\sum_{r=0}^{n/2} \frac{\prod_{i=1}^{r} (1-p^{2i-1}Y^2) \prod_{i=1}^{(n-2r-2)/2}(1-p^{-2i-n-1}Y^2t^4)}{(1-p^{(-n-3)/2}tY^2)\phi_{(n-2r)/2}(p^{-2})} \\
& \phantom{xxxxxxxxxxxxx}\times \sum_{d \in {\mathcal U}(n,n-2r,d_0)} D_{2r}(d_0,d,Y,t)\widetilde Q_{2r}^{(0)}(n;d_0,\varepsilon,1, X,Y,t).
\end{align*}
\end{enumerate}
\end{thms}
\begin{proof} Let $p\not=2.$ Let $B$ be a symmetric matrix of degree $2r$ or $2r+1$ with entries in $\ZZ_p.$ Then we note that $\Theta_{n-2r,d} \bot pB$ belongs to ${\mathcal L}_{n,p}(d_0)$ if and only if $B \in S_{2r,p}(d_0d) \cap S_{2r,p}$, and that  $\Theta_{n-2r-1,d} \bot pB$ belongs to ${\mathcal L}_{n,p}(d_0)$ if and only if $B \in S_{2r+1,p}(p^{-1}d_0d) \cap S_{2r+1,p}.$ Thus by the theory of Jordan decompositions, for $\omega=\varepsilon^l$ we have
 \begin{align*}
& \widetilde R_{n}(d_0,\omega,X,Y,t)=\kappa(d_0,n,l)^{-1} Y^{\nu(d_0)/2}\\
& \times \left\{ \sum_{r=0}^{n/2} \sum_{d \in {\mathcal U}(n,n-2r,d_0)} \sum_{B' \in S_{2r,p}(d_0d)} 
   \frac{G_p^{(0)}(\Theta_{n-2r,d} \bot pB',p^{-(n+1)/2}Y)}{\alpha_p(\Theta_{n-2r,d} \bot pB') } \right. \\
& \times B_p^{(0)}(\Theta_{n-2r,d} \bot pB',p^{-n/2-3/2}Yt^2) \widetilde G_p^{(0)}(\Theta_{n-2r,d} \bot pB',1,X,p^{-n-1}t^2Y) \omega(\Theta_{n-2r,d} \bot pB') (tY^{-1/2})^{\nu(\det (pB'))} \\
& +\sum_{r=0}^{(n-2)/2} \sum_{d \in {\mathcal U}(n,n-2r-1,d_0)} \sum_{B' \in S_{2r+1,p}(p^{-1}d_0d)}
  \frac{G_p^{(0)}(\Theta_{n-2r-1,d} \bot pB',p^{-(n+1)/2}Y)}{\alpha_p(\Theta_{n-2r-1,d} \bot pB') } \\
& \times  B_p^{(0)}(\Theta_{n-2r-1,d} \bot pB',p^{-n/2-3/2}Yt^2) \widetilde G_p^{(0)}(\Theta_{n-2r-1,d} \bot pB',1,X,p^{-n-1}t^2Y) \\
& \times \left. \omega(\Theta_{n-2r-1,d} \bot pB')(tY^{-1/2})^{\nu(\det (pB'))} \right\}.
\end{align*}
By Lemmas \ref{lem.primitive-Siegel-series} and \ref{lem.Andrianov-B} we have
  \begin{align*}
&G_p^{(0)}(\Theta_{n-2r,d} \bot pB',p^{-(n+1)/2}Y)B_p^{(0)}(\Theta_{n-2r,d} \bot pB',p^{-n/2-3/2}Yt^2)\\
&=((1+\xi_0 p^{-1/2}Y)(1-p^{-(n+3)/2}t^2Y))^{-1} \prod_{i=m_0}^{r-1} (1-p^{2i-1}Y^2) \\\
& \times \prod_{i=1}^{(n-2r)/2}(1-p^{-2i-n-1}Y^2t^4)(1+p^{r-1/2}\chi(d) Y)(1-p^{-n-3/2+r}\chi(d)Yt^2),
\end{align*}
  and 
 \begin{align*}
&G_p^{(0)}(\Theta_{n-2r-1,d} \bot pB',p^{-(n+1)/2}Y)B_p^{(0)}(\Theta_{n-2r-1,d} \bot pB',p^{-n/2-3/2}Yt^2)\\
 &=(1-\xi_0 p^{-1/2}Y)(1-p^{-(n+3)/2}t^2Y)^{-1} \prod_{i=m_0}^{r} (1-p^{2i-1}Y^2) \prod_{i=1}^{(n-2r)/2}(1-p^{-2i-n-1}Y^2t^4).
\end{align*}
   Thus the assertion follows from \cite[Lemma 5.8]{KK15}, \cite[Lemma 4.3.2]{KK14}, and \cite[Propositions 4.3.3 and 4.3.4]{KK14}.

Let $p=2$. Then, similarly to above we have
 \begin{align*}
& \widetilde R_{n}(d_0,\omega,X,Y,t)=\kappa(d_0,n,l)^{-1} Y^{\nu(d_0)/2}\\
& \times \Bigl\{ \sum_{r=0}^{n/2} \sum_{d \in {\mathcal U}(n,n-2r,d_0)} \sum_{B' \in S_{2r,2}(d_0d) \cap S_{2r,2,e}} 
   \frac{G_2^{(0)}(\Theta_{n-2r,d} \bot 2B',p^{-(n+1)/2}Y)}{ \alpha_2(\Theta_{n-2r,d} \bot 2B') } \\
& \times B_2^{(0)}(\Theta_{n-2r,d} \bot 2B',2^{-n/2-3/2}Yt^2) \widetilde G_2^{(0)}(\Theta_{n-2r,d} \bot 2B',1,X,p^{-n-1}t^2Y) \omega(\Theta_{n-2r,d} \bot 2B') (tY^{-1/2})^{\nu(\det (2B'))} \\
& +\sum_{r=0}^{(n-2)/2}  \sum_{B' \in S_{2r+2,p}(d_0) \cap S_{2r+2,2,o}}
  \frac{G_2^{(0)}(\Theta_{n-2r-2} \bot 2B',2^{-(n+1)/2}Y)}{ \alpha_2(\Theta_{n-2r-2} \bot 2B') } \\
& \times  B_2^{(0)}(\Theta_{n-2r-2} \bot pB',2^{-n/2-3/2}Yt^2) \widetilde G_2^{(0)}(\Theta_{n-2r-2} \bot 2B',1,X,2^{-n-1}t^2Y) \\
& \times  \omega(\Theta_{n-2r-2,d} \bot 2B')(tY^{-1/2})^{\nu(\det (2B'))} \Bigr \}.
\end{align*}
Here we make the convention that we have $\Theta_{n-2r-2,d} \bot 2B'=2B'$ if $r=(n-2)/2$.
Then the assertion can be proved similarly to above by using Lemmas \ref{lem.primitive-Siegel-series} and \ref{lem.Andrianov-B}, \cite[Lemma 5.8]{KK15}, \cite[Lemma 4.3.2]{KK14}, and \cite[Propositions 4.3.3 and 4.3.4]{KK14}.
 \end{proof}
  
 
Now to rewrite the above theorem, first we express $\widetilde P_{m-1}^{(0)}(n+1;d_0,\omega, \eta, X,Y,t)$ in terms of  $Q_{2r+1}^{(1)}(n+1;d_0, \omega, \eta , X,Y,t)$ and $Q_{2r}^{(0)}(n+1;d_0 d,\omega,\eta, X,Y,t).$ First we recall the following result 
 (cf. \cite{KK15}, Corollary 5.12).

\begin{props}
\label{prop.express-Q-as-P}
Let $r$ be a non-negative integer. Let $d_0$ be an element of ${\mathcal F}_{p}$ and $\xi=\pm 1.$ Then for any non-negative integer $a$, the following assertions hold.
 \begin{itemize}
\item[(1)] Let $l=0$ or $\nu(d_0)=0.$ Then  
 $$Q_{2r}^{(0)}(a;d_0,\varepsilon^l,\xi, X,Y,t)=\sum_{m=0}^r \sum_{d \in {\mathcal U}(2r,2m,d_0)} \frac{(-1)^m (\chi(d) +p^{-m})p^{-m^2}}{2^{1-\delta_{0,r-m}+\delta_{0,r}}\phi_m(p^{-2}) }\widetilde P_{2r-2m}^{(0)}(a;d_0 d , \varepsilon^l,\xi \chi(d), X,Y,t)$$
$$+ \sum_{m=0}^{r-1}  \frac{(-1)^{m+1} p^{-m-m^2}}{\phi_m(p^{-2}) }\widetilde P_{2r-2m-1}^{(1)}(a;d_0 , \varepsilon^l,\xi , X,Y,t)),$$
\begin{eqnarray*}
&& Q_{2r+1}^{(1)}(a;d_0,\varepsilon^l,\xi,X,Y,t)=\sum_{m=0}^r  \frac{(-1)^m p^{-m-m^2}}{\phi_m(p^{-2}) }\widetilde P_{2r+1-2m}^{(1)}(a;d_0,\varepsilon^l,\xi,X,Y,t) \\
&& \phantom{xxxxxxxxxxxxxxx} + \sum_{m=0}^{r} \sum_{d \in {\mathcal U}(2r+1,2m+1,d_0)} \frac{(-1)^{m+1} p^{-m-m^2}}{ 2^{1-\delta_{0,r-m}}\phi_m(p^{-2}) }\widetilde P_{2r-2m}^{(0)}(a;d_0 d ,\varepsilon^l,\xi \chi(d),X,Y,t)).
\end{eqnarray*}
\item[(2)] Let $\nu(d_0) >0.$ We have 
 $$Q_{2r+1}^{(1)}(n;d_0,\varepsilon,\xi,X,Y,t)=\sum_{m=0}^r  \frac{(-1)^m p^{m-m^2}}{\phi_m(p^{-2}) }\widetilde P_{2r+1-2m}^{(1)}(n;d_0,\varepsilon,\xi,X,Y,t),$$
and
 $$Q_{2r}^{(0)}(n;d_0,\varepsilon,\xi,X,Y,t)=0.$$
\end{itemize}
\end{props}

The following lemma is well known (cf. \cite[Lemma 5.13]{KK15}).


\begin{lems}
\label{lem.q-identity}
Let $l$ be a positive integer, and $q,U$ and $Q$ variables. Then 
\begin{eqnarray*}
\prod_{i=1}^l (1-U^{-1}Qq^{-i+1})U^l=\sum_{m=0}^l \frac{\phi_l(q^{-1})}{\phi_{l-m}(q^{-1}) \phi_m(q^{-1})}\prod_{i=1}^{l-m} (1-Qq^{-i+1})\prod_{i=1}^m (1-Uq^{i-1}) (-1)^mq^{(m-m^2)/2}.
\end{eqnarray*}
\end{lems}

\begin{thms}  
\label{explicit-R}
  Let the notation be as in Theorem \ref{th.formula-for-R}. \\
\begin{itemize}
\item[(1)] Suppose  that $\nu(d_0)=0$ and 
put $\xi_0=\chi(d_0).$ Then   
\begin{align*}
& \widetilde R_{n}(d_0,\omega,X,Y,t)=(1-p^{-n-1}t^2) \times \Big\{\sum_{l=0}^{n/2}\sum_{d \in {\mathcal U}(n,n-2l,d_0)}\frac{ \widetilde P_{2l}^{(0)}(n;d_0d,\omega,\chi(d),X,Y,t) T_{2l}(d_0,d,Y,t)}{ 2^{1-\delta_{0,l}}} \\
& \times  \frac{\prod_{i=1}^{(n-2-2l)/2}(1-p^{-2l-n-2i-2}t^4) (p^{2l-1}Y)^{n/2-l}\prod_{i=0}^{l-1}(1-p^{2i-1}Y^2)}{ (1+p^{-1/2}\xi_0Y)\phi_{n/2-l}(p^{-2})} \\
& -\sum_{l=0}^{(n-2)/2}  \widetilde P_{2l+1}^{(1)}(n;d_0,\omega,1,X,Y,t) \\
& \times \frac{\prod_{i=1}^{(n-2-2l)/2}(1-p^{-2l-n-2i-2}t^4)(p^{2l+1}Y)^{n/2-l}p^{-n/2+1/2}\prod_{i=0}^l (1-p^{2i-1}Y^2)}{ (1+p^{-1/2}\xi_0Y)\phi_{n/2-l-1}(p^{-2})} \Big\},
\end{align*}
where
$$T_{2l}(d,Y,t)=(1+p^{-n/2+l}\chi(d)) (1+p^{-n/2-l-1}t^2\chi(d))(1+p^{l-1/2}\chi(d)Y).
$$ 
\item[(2)] Suppose that $\nu(d_0) >0$.
\begin{itemize}
\item[(2.1)] Suppose that $\omega=\iota$. Then
\begin{align*}
& \widetilde R_{n}(d_0,\omega,X,Y,t)=(1-p^{-n-1}t^2) \times \Big\{\sum_{l=1}^{n/2}\sum_{d \in {\mathcal U}(n,n-2l,d_0)}\frac{ \widetilde P_{2l}^{(0)}(n;d_0d,\omega,\chi(d),X,Y,t) T_{2l}(d_0,d,Y,t)} { 2} \\
& \times  \frac{\prod_{i=1}^{(n-2-2l)/2}(1-p^{-2l-n-2i-2}t^4) (p^{2l-1}Y)^{n/2-l}\prod_{i=1}^{l-1}(1-p^{2i-1}Y^2)}{\phi_{n/2-l}(p^{-2})} \\
&-\sum_{l=0}^{(n-2)/2}  \widetilde P_{2l+1}^{(1)}(n;d_0,\omega,1,X,Y,t) \\
& \times \frac{\prod_{i=1}^{(n-2-2l)/2}(1-p^{-2l-n-2i-2}t^4)(p^{2l+1}Y)^{n/2-l}p^{-n/2+1/2}\prod_{i=1}^l (1-p^{2i-1}Y^2)}{ \phi_{n/2-l-1}(p^{-2})} \Big\}.
\end{align*} 
\item[(2.2)] Suppose that  $\omega=\varepsilon.$ Then 
$\widetilde R_{n}(d_0,\omega,X,Y,t)=0.$
\end{itemize}
\end{itemize}
\end{thms}
\begin{proof} (1) By Theorem \ref{th.formula-for-R} and Proposition \ref{prop.express-Q-as-P}, we have 
\begin{align*}
& \widetilde R_{n}(d_0,\omega;X,Y,t) =  \sum_{r=0}^{n/2}  \frac{\prod_{i=0}^{r-1} (1-p^{2i-1}Y^2) \prod_{i=1}^{(n-2r)/2}(1-p^{-2i-n-1}Y^2t^4)}{(1+p^{-1/2}\xi_0Y)(1-p^{(-n-3)/2}t^2Y)\phi_{(n-2r)/2}(p^{-2})}\\
& \times \sum_{d_1 \in {\mathcal U}(n,n-2r,d_0)} \frac{D_{2r}(d_1,Y,t)}{2^{1-\delta_{0,r}}} \Bigg\{\sum_{m=0}^r \sum_{d_2 \in {\mathcal U}(2r,2m,d_0d_1)} \frac{(-1)^m (\chi_p(d_2) +p^{-m})p^{-m^2}}{2^{1-\delta_{0,r-m}+\delta_{0,r}}\phi_m(p^{-2}) } \\
& \times \widetilde P_{2r-2m}^{(0)}(n;d_0 d_1 d_2, \omega, \chi(d_1)\chi_p(d_2), X,Y;t) \\
& + \sum_{m=0}^{r-1} \frac{(-1)^{m+1} p^{-m-m^2}}{\phi_m(p^{-2}) }\widetilde P_{2r-2m-1}^{(1)}(n;d_0d_1,\omega, \chi(d_1), X,Y;t))\Bigg\} \\
& + \sum_{r=0}^{(n-2)/2}  \frac{\prod_{i=0}^{r} (1-p^{2i-1}Y^2) \prod_{i=1}^{(n-2r)/2}(1-p^{-2i-n-1}Y^2t^4)}{(1+p^{-1/2}\xi_0Y)(1-p^{(-n-3)/2}t^2Y) \phi_{(n-2r-2)/2}(p^{-2})} \\ 
& \times \{\sum_{m=0}^r  \frac{(-1)^m p^{-m}p^{-m^2}}{ \phi_m(p^{-2}) }\widetilde P_{2r+1-2m}^{(1)}(n;d_0 ,\omega,1, X,Y,t) \\
& + \sum_{m=0}^{r-1} \sum_{d_2 \in {\mathcal U}(2r+1,2m+1,d_0)} \frac{(-1)^{m+1} p^{-m-m^2}}{2^{1-\delta_{0,r-m}}\phi_m(p^{-2}) }\widetilde P_{2r-2m}^{(0)}(n;d_0  d_2, \omega, \chi_p(d_2),X,Y,t)\}.
\end{align*}
We note that by Proposition \ref{prop.local-modified-KM} and Theorem \ref{th.explicit-local-KM}, for any $d_1 \in {\mathcal U}$ we have 
\begin{align*}
\widetilde P_{2r+1-2m}^{(1)}(n;d_0 d_1, \omega,\chi(d_1),X,Y,t)=\widetilde P_{2r+1-2m}^{(1)}(n;d_0, \omega,1,X,Y,t).
\end{align*}
We also note that $\calu(2l+2m+1,2m+1,d_0)=\calu(n.n-2l,d_0)$ for any $0 \le l \le (n-2)/2$ and $0 \le m \le (n-2)/2-l$.
Hence we have
\begin{align*}
& \widetilde R_{n}(d_0,\omega;X,Y,t)= \sum_{l=0}^{n/2}  \sum_{d \in {\mathcal U}(n,n-2l,d_0)} \frac{\widetilde P_{2l}^{(0)}(n;d_0 d,\omega,\chi(d),X,Y,t)}{2^{1-\delta_{0,l}}}\\
& \times \Big\{\sum_{m=0}^{(n-2l)/2} \sum_{d_1 \in {\mathcal U}(n-2l,n-2l-2m,d)} \frac{D_{2l+2m}(d_1,Y,t)}{2}(\chi(d_1)\chi(d) +p^{-m})(-1)^mp^{-m^2} \nonumber \\
& \times \frac{\prod_{i=0}^{l+m-1}(1-p^{2i-1}Y^2) \prod_{i=1}^{(n-2l-2m)/2}(1-p^{-2i-n-1}Y^2t^4)}{(1+p^{-1/2}\xi_0Y)(1-p^{(-n-3)/2}t^2Y) \phi_m(p^{-2})\phi_{(n-2l-2m)/2}(p^{-2})}\nonumber \\
& -\sum_{m=0}^{n/2-l-1} (-1)^mp^{-m-m^2}\frac{\prod_{i=0}^{l+m} (1-p^{2i-1}Y^2) \prod_{i=1}^{(n-2l-2m)/2}(1-p^{-2i-n-1}Y^2t^4)} { (1+p^{-1/2}\xi_0Y)(1-p^{(-n-3)/2}t^2Y) \phi_m(p^{-2})\phi_{(n-2-2l)/2-m}(p^{-2})} \nonumber \Big\} \\
& +\sum_{l=0}^{(n-2)/2}  \widetilde P_{2l+1}^{(1)}(n;d_0, \omega, 1, X,Y,t) \nonumber \\
& \times \Big\{\sum_{m=0}^{(n-2-2l)/2}  (-1)^mp^{-m-m^2} \frac{\prod_{i=0}^{l+m} (1-p^{2i-1}Y^2) \prod_{i=1}^{(n-2l-2m-2)/2}(1-p^{-2i-n-1}Y^2t^4)}{(1+p^{-1/2}\xi_0Y)(1-p^{(-n-3)/2}t^2Y) \phi_m(p^{-2})\phi_{(n-2-2l)/2-m}(p^{-2})} \nonumber\\
& -\sum_{m=0}^{(n-2-2l)/2}   \sum_{d_1 \in {\mathcal U}(n-2l,n-2l-2m-2,d_0)}\frac{D_{2l+2m+2}(d_1,Y,t)}{2}(-1)^m p^{-m-m^2} \nonumber \\
& \times \frac{\prod_{i=0}^{l+m} (1-p^{2i-1}Y^2) \prod_{i=1}^{(n-2l-2m-2)/2}(1-p^{-2i-n-1}Y^2t^4)}{ (1+p^{-1/2}\xi_0Y)(1-p^{(-n-3)/2}t^2Y) \phi_m(p^{-2})\phi_{(n-2-2l)/2-m}(p^{-2})}\Bigr\}.\nonumber\\
\end{align*}
For $d \in {\mathcal U}(n,n-2l,d_0)$, we have
\begin{align*}
& \sum_{d_1 \in {\mathcal U}(n-2l,n-2l-2m,d)}  \frac{D_{2l+2m}(d_1,Y,t)(\chi(d_1)\chi(d)+p^{-m})}{2}-(1-p^{2l+2m-1}Y^2)(1-p^{-n+2m+2l}) p^{-m} \\
& =p^{-n+m+2l}(1-p^{2l+2m-1}Y^2)(1+p^{-n/2-l-1}\chi(d)t^2)\\
& +p^{-n/2+l+m}\chi(d)(1-p^{-n-1}t^2)(1+p^{l-1/2}\chi(d)Y)(1+p^{-1/2+n/2}Y),
\end{align*}
and 
\begin{align*}
& 1-p^{-2n+2l+2m-1}t^4Y^2-\sum_{d_1 \in {\mathcal U}(n-2l,n-2l-2m-2,d_0)}\frac{D_{2l+2m+2}(d_1,Y,t)}{2}\\
&\phantom{xxxxxxxxxxxxxxxxxxxxxx} =-Yp^{-n/2+2m+2l+3/2}(1-p^{-n-1}t^2)(1-p^{(-n-3)/2}t^2Y).
\end{align*}
Hence 
\begin{align*}
& \widetilde R_{n}(d_0,\omega,X,Y,t) = (1-p^{(-n-3)/2}t^2Y)^{-1}\sum_{l=0}^{n/2}\sum_{d \in {\mathcal U}(n,n-2l,d_0)}\frac{\widetilde P_{2l}^{(0)}(n;d_0 d,\omega,\chi(d),X,Y,t) }{2^{1-\delta_{0,l}}}\\
&\phantom{xxxxxxxxxxxxx} \times \{ p^{-n+2l} (1+p^{-n/2-l-1}t^2\chi(d)) \frac{\prod_{i=0}^{l} (1-p^{2i-1}Y^2) }{1+p^{-1/2}\xi_0Y}\\
&\phantom{xxxxx} \times \sum_{m=0}^{n/2-l}\frac{\prod_{i=1}^{m} (-1)^mp^{m-m^2}(1-p^{2l+1}p^{2i-2}Y^2) \prod_{i=1}^{n/2-l-m}(1-p^{-2i-1-n}Y^2t^4)}{ \phi_m(p^{-2})\phi_{n/2-l-m}(p^{-2})} \\
&\phantom{xxxxx} +p^{-n/2+l}\chi(d)(1-p^{-n-1}t^2)(1+p^{l-1/2}\chi(d)Y)(1+p^{n/2-1/2}Y)\frac{\prod_{i=0}^{l-1} (1-p^{2i-1}Y^2)}{1+p^{-1/2}\xi_0Y}\\
&\phantom{xxxxx} \times \sum_{m=0}^{n/2-l} \frac{\prod_{i=1}^{m} (-1)^mp^{m-m^2}(1-p^{2l-1}p^{2i-2}Y^2) \prod_{i=1}^{n/2-l-m}(1-p^{-2i-1-n}Y^2t^4)}{ \phi_m(p^{-2})\phi_{n/2-l-m}(p^{-2})}\} \\
& \phantom{xxxxx}- (1-p^{-n-1}t^2) \sum_{l=0}^{n/2-1}\widetilde P_{2l+1}^{(1)}(n;d_0 ,\omega,1,X,Y,t) p^{-n/2+2l+3/2}Y\frac{\prod_{i=0}^{l} (1-p^{2i-1}Y^2)}{1+p^{-1/2}\xi_0Y}\\
& \phantom{xxxx} \times  \sum_{m=0}^{n/2-l-1} \frac{\prod_{i=1}^{m} (-1)^mp^{m-m^2}(1-p^{2l+1}p^{2i-2}Y^2) \prod_{i=1}^{n/2-l-m-1}(1-p^{-2i-1-n}Y^2t^4)}{\phi_m(p^{-2})\phi_{n/2-l-m-1}(p^{-2})}.
\end{align*}
Then by Lemma \ref{lem.q-identity}, we have
\begin{align*}
	& \widetilde R_{n}(d_0,\omega,X,Y,t) = (1-p^{(-n-3)/2}t^2Y)^{-1}\sum_{l=0}^{n/2}\sum_{d \in {\mathcal U}(n,n-2l,d_0)}\frac{\widetilde P_{2l}^{(0)}(n;d_0 d,\omega,\chi(d),X,Y,t) }{2^{1-\delta_{0,l}}}\\
&\phantom{xxxxx} \times \Bigl\{ p^{-n+2l}  (1+p^{-n/2-l-1}t^2\chi(d))\frac{\prod_{i=0}^{l} (1-p^{2i-1}Y^2)}{(1+p^{-1/2}\xi_0Y)\phi_{n/2-l}(p^{-2})} \\
&\phantom{xxxxx} \times (p^{2l+1}Y)^{n/2-l} \prod_{i=1}^{n/2-l}(1-p^{-2l-n-2i-2}t^4)\\
&\phantom{xxxxx} +p^{-n/2+l}\chi(d)(1-p^{-n-1}t^2)(1+p^{l-1/2}\chi(d)Y)(1+p^{n/2-1/2}Y) \\
&\phantom{xxxxx} \times \frac{\prod_{i=0}^{l-1} (1-p^{2i-1}Y^2)}{(1+p^{-1/2}\xi_0Y)\phi_{n/2-l}(p^{-2})} (p^{2l-1}Y)^{n/2-l} \prod_{i=1}^{n/2-l}(1-p^{-2l-n-2i}t^4)\Bigr\} \\
&\phantom{xxxxx} -(1-p^{-n-1}t^2) \sum_{l=0}^{n/2-1} \widetilde P_{2l+1}^{(1)}(n;d_0 ,\omega,1,X,Y,t)\\
&\phantom{xxxxx} \times \frac {p^{-n/2+2l+3/2}Y}{\phi_{n/2-l-1}(p^{-2})} \frac{\prod_{i=0}^{l} (1-p^{2i-1}Y^2)}{1+p^{-1/2}\xi_0Y}(p^{2l+1}Y)^{n/2-l-1}\prod_{i=1}^{n/2-l-1}(1-p^{-2l-n-2i-2}t^4).
\end{align*}
Thus by a simple computation we prove the assertion.

(2) The assertion (2.1) can be proved in the same way as above remarking that $\chi(d_0)=0$ and $\calu(n,n,d_0)=\emptyset$. The assertion (2.2) follows from (2) of Theorem \ref{th.formula-for-R} and (2) of Proposition \ref{prop.express-Q-as-P}.
\end{proof}

By Proposition \ref{prop.local-modified-KM} and Theorem 
\ref{th.explicit-local-KM}, we immediately obtain:


\begin{cors}
\label{cor.explicit-R}
  Let the notation be as in Theorem \ref{th.formula-for-R}.  Suppose  that $\nu(d_0)=0$ or $\omega=\iota.$ 
Put $\xi_0=\chi(d_0).$  Then   
\begin{align*}
&\widetilde R_{n}(d_0,\omega,X,Y,t)=Y^{\nu(d_0)/2}(1-p^{-n-1}t^2)\prod_{i=1}^{(n-2)/2 }(1-p^{-2n+2i-2}t^4) \\
&\phantom{xxxxxxxxxx} \times \Bigl( \sum_{l=0}^{n/2}\frac{\prod_{i=1}^{l}(1-p^{-n-2l-3+2i}t^4) (p^{2l-1}Y)^{n/2-l}\prod_{i=m_0}^{l-1}(1-p^{2i-1}Y^2)}{\phi_{n/2-l}(p^{-2})(1+p^{-1/2}\xi_0 Y)} \\
&\phantom{xxxxxxxxxx} \times \sum_{d \in {\mathcal U}(n,n-2l,d_0)} T_{2l}(d,Y,t) \frac{P_{2l}^{(0)}(d_0d,\omega,\chi(d),X,tY^{-1/2})}{2^{1-\delta_{0,l}}}  \\
&\phantom{xxxxxxxxxx} -\sum_{l=0}^{(n-2)/2}  \frac{\prod_{i=1}^{l}(1-p^{-n-2l-3+2i}t^4) (1-p^{-1/2}\xi_0 Y)}{\phi_{n/2-l-1}(p^{-2})} \\
&\phantom{xxxxxxxxxx}  \times \prod_{i=1}^{l-1}(1-p^{2i-1}Y^2) (p^{2l+1}Y)^{n/2-l} p^{-n/2+1/2}  P_{2l+1}^{(1)}(d_0,\omega,1,X,tY^{-1/2})\Bigr).
\end{align*}
\end{cors}
\subsubsection{Explicit formulas of formal power series of Rankin-Selberg type}
 
We give the following result, which is one of key ingredients for proving our main result.
 
\begin{thms}    
\label{th.explicit-local-factor} 
 Let $d_0 \in {\mathcal F}_p$ and put $\xi_0=\chi(d_0).$
\begin{itemize}
\item[(1)]  We have
 \begin{align*}
&H_{n}(d_0,\iota,X,Y,t)=\phi_{(n-2)/2}(p^{-2})^{-1}(1-p^{-n/2}\xi_0)^{-1}(p^{-1}t)^{\nu(d_0)}(1-p^{-n-1}t^2)\prod_{i=1}^{\tfrac n2-1} (1-p^{-2n+2i-2}t^4) \\
&\times \frac{L_p(\xi_0;X+X^{-1},Y+Y^{-1},p^{(n-1)/2}+p^{(-1+p)/2},p^{-n/2-3/2}t^2)}{\prod_{a,b=\pm 1}  (1-p^{-2}X^aY^bt^2) \prod_{a,b=\pm 1}(1-p^{-n-1}X^aY^bt^2)}\times \frac 1{\prod_{i=1}^{\tfrac n2-1} \prod_{a,b=\pm 1}(1-p^{-2i-1}X^aY^bt^2)}.
\end{align*}
 \item[(2)]  If $\nu(d_0)> 0$, then $H_{n}(d_0,\varepsilon,X,Y,t)=0.$ If $\nu(d_0)=0$, then we have 
 \begin{align*}
&H_{n}(d_0,\varepsilon,X,Y,t)=\phi_{(n-2)/2}(p^{-2})^{-1}(1-p^{-n/2}\xi_0)^{-1}\\
&\phantom{xxxxxxxx} \times (1-p^{-n-1}t^2)\prod_{i=1}^{n/2-1} (1-p^{-2n+2i-2}t^4)\times\frac{1+\xi_0 p^{-n/2-1} t^2}{   \prod_{i=1}^{n/2} \prod_{a,b=\pm 1} (1-p^{-2i}X^aY^bt^2)}.
\end{align*}
\end{itemize}
 \end{thms}
\begin{proof} First suppose that $\omega=\iota.$ For an integer $l$, put
\begin{align*}
&V(X,Y,t)=(1-t^2p^{-2}XY^{-1})(1-t^2p^{-2}X^{-1}Y^{-1}) \times \prod_{i=1}^{n/2}(1-t^2p^{-2i-1}XY^{-1})(1-t^2p^{-2i-1}X^{-1}Y^{-1}).
\end{align*}
Then  by Theorem \ref{th.explicit-local-KM}, and Corollary \ref{cor.explicit-R}, we have
$$\widetilde R_{n}(d_0,\iota,X,Y,t)=\frac{(1-p^{-n-1}t^2) \prod_{i=1}^{n/2}(1-p^{-n-2i+2}t^4) S(d_0,\iota, X,Y,t)}{\phi_{(n-2)/2}(p^{-2}) (1-p^{-n/2}\xi_0) V(X,Y,t)},
$$
where $S(d_0, \iota,X,Y,t)$ is a polynomial in  $t$ of degree at most $2n+6$ such that 
\begin{align*} 
 &S(d_0, \iota,X,Y,t) =(1-p^{-1/2}\xi_0Y)(1+p^{n/2-1/2}Y) \\
& \times \{(1+p^{-n/2-3/2}Y^{-1} t^2)(1+p^{-n/2-5/2}Y^{-1} t^2\xi_0^2) -\xi_0t^2Y^{-1}p^{-n/2-2}(X+X^{-1}+p^{1/2-n/2}+p^{-1/2+n/2})\}  \\
& \times (1+p^{-n-1}t^2) \prod_{i=1}^{n/2-1} (1-p^{2i-1}Y^2)\prod_{i=1}^{n/2} (1-p^{2i-3-2n}t^4) \\
& +(1-p^{-n-1}XY^{-1}t^2)(1-p^{-n-1}X^{-1}Y^{-1}t^2) U(d_0,X,Y,\iota,t),
\end{align*}
with $U(d_0,\iota,X,Y,t)$ a polynomial in $t.$
Hence  by Theorem \ref{th.local-RS-series} we have 
\begin{align*}
&H_{n}(d_0,\iota,X,Y,t)=\frac{1}{(1-p^{-n/2}\xi_0) \phi_{(n-2)/2}(p^{-2})} (1-p^{-n-1}t^2)\prod_{i=1}^{n/2} (1-p^{-2n+2i-2}t^4) \\
&\phantom{xxxxxxxxxxxx} \times \frac{S(d_0,\iota,X,Y,t)}{ \prod_{a,b=\pm 1}(1-p^{-2}X^aY^bt^2) }\times \frac{1}{ \prod_{i=1}^{n/2} \prod_{a,b=\pm 1}(1-p^{-2i-1}X^aY^bt^2)} \\
&\phantom{xxxxxxxxxxxx} \times \frac{1}{\prod_{i=1}^{n/2}(1-p^{-2i}XYt^2)(1-p^{-2i}X^{-1}Yt^2)}.
\end{align*}
Hence the power series $\widetilde R_{n-1}(d_0,\iota,X,Y,t)$ is a rational function of $X,Y$ and $t$, and is invariant under the transformation $Y \mapsto Y^{-1}$. This implies that the reduced denominator of the rational function  $H_{n}(d_0,\iota,X,Y,t)$ in $t$ is at most 
 \begin{align*}
&\prod_{a,b=\pm 1} (1-p^{-2}X^aY^bt^2) \prod_{i=1}^{n/2} \prod_{a,b=\pm 1} (1-p^{-2i-1}X^aY^bt^2)
\end{align*}
 and therefore we have

$$
 S(d_0,\iota,X,Y,t)=t^{\nu(d_0)}T(X,Y,t^2) \prod_{i=1}^{(n-2)/2}(1-p^{-2i-2}XYt^2)(1-p^{-2i-2}X^{-1}Yt^2),
$$
where $T(X,Y,u)$ is a polynomial in $u$ of degree at most $5$ with coefficients in $\QQ[X+X^{-1},Y+Y^{-1}].$
Assume that $\nu(d_0)=0.$ Then the degree of $T(X,Y,u)$ is $5$, and we  easily see that
the constant term is $1$ and, the $5$-th  coefficient of $T(X,Y,u)$ is $p^{-5/2-9}.$ Hence 
 $T(X,Y,u)$ can be expressed as 
 $$T(X,Y,u)=(1+p^{3n/2-5}\xi_0u)\prod_{i,j=\pm 1} (1-p^{-n-1}X^iY^ju)+G(X,Y,u),$$
 where $G(X,Y,u)$ is a polynomial of $u$ of degree at most 4 with coefficients in $\QQ[X+X^{-1},Y+Y^{-1}]$
 such that $G(X,Y,0)=0.$ 
 We have
\begin{eqnarray*}
&& G(X,Y,p^{n+1}X^iY)=(1-p^{n/2-1}\xi_0)(1-p^{n-1}X^{2i}Y^2)(1+X^iY) \\
&& \phantom{xxxxxxxxxx} \times (1-p^{-1/2}\xi_0X^i)(1+p^{n/2-1/2}X^i)(1-p^{-1/2}\xi_0Y)(1+p^{n/2-1/2}Y),
\end{eqnarray*}
for $i=\pm 1.$ The polynomial $G(X,Y,t)$ is invariant under the transformation $Y \mapsto Y^{-1}.$ Hence we have
\begin{eqnarray*}
&& G(X,Y,p^{n+1}X^iY^{-1})=(1-p^{n/2-1}\xi_0)(1-p^{n-1}X^{2i}Y^{-2})(1+X^iY^{-1})\\
&& \phantom{xxxxxxxxxx} \times (1-p^{-1/2}\xi_0X^i)(1+p^{n/2-1/2}X^i)(1-p^{-1/2}\xi_0Y^{-1})(1+p^{n/2-1/2}Y^{-1}),
\end{eqnarray*}
for $i=\pm 1.$ Then we have 
\begin{eqnarray*}
&& G(X,Y,u)=(1-p^{n/2-1}\xi_0)u \sum_{i,j=\pm} \frac{\prod_{(a,b) \not=(-i,-j)} (1-p^{-n-1}X^aY^b u)}{p^{n+1}X^iY^j (1-X^{2i})(1-Y^{2j})(1-X^{2i}Y^{2j})}\\
&& \times (1-p^{-1/2}\xi_0X^i)(1+p^{n/2-1/2}X^i)(1-p^{-1/2}\xi_0Y^j)(1+p^{n/2-1/2}Y^j) (1-p^{n-1}X^{2i}Y^{2j})(1+X^iY^j).
\end{eqnarray*}

We define a rational function  $\widetilde L_p(d_0; X,Y,u)$ in $u,X,Y$ as 
\begin{eqnarray*}
&& \widetilde L_p(d_0;X,Y,u)=(1+p^{3n/2-5}u\xi_0) \prod_{i,j=\pm 1} (1-p^{-n-1}X^iY^j u) \\
&& \phantom{xxxxxsssssss} +(1-p^{n/2-1}\xi_0)u \sum_{i,j=\pm 1} \frac{\prod_{(a,b) \not=(-i,-j)} (1-p^{-n-1}X^aY^b u)}{p^{n+1}X^iY^j (1-X^{2i})(1-Y^{2j})(1-X^{2i}Y^{2j})}\\
&&\times (1-p^{-1/2}\xi_0X^i)(1+p^{n/2-1/2}X^i)(1-p^{-1/2}\xi_0Y^j)(1+p^{n/2-1/2}Y^j) (1-p^{n-1}X^{2i}Y^{2j})(1+X^iY^j).
\end{eqnarray*}
Then we have
\[T(X,Y,u)=\widetilde L_p(d_0;X,Y,u).\]
Then by a computation with Mathematica, we have
\[\widetilde L_p(d_0; X,Y,u)=L_p(\xi_0,X+X^{-1},Y+Y^{-1},p^{(n-1)/2}+p^{(1-n)/2},p^{-n/2-3/2}u).\]
This proves the assertion in the case $\nu(d_0)=0.$ Next assume that $\nu(d_0)>0.$ Then 
the degree of $T(X,Y,u)$ is $4,$ and by the same argument as above we see that we have
$$T(X,Y,u)=L_p(\xi_0;X+X^{-1},Y+Y^{-1},p^{(n-1)/2}+p^{(1-n)/2},p^{-n/2-3/2}u).$$
Similarly the assertion for $\nu(d_0)=0$ and $\omega=\varepsilon$ can be proved. 
 Next suppose that $\nu(d_0) >0$ and $\omega=\varepsilon.$ 
Then the assertion follows from  Theorem \ref{th.local-RS-series}  and (2) of  Theorem \ref{th.formula-for-R}.
\end{proof}

\subsection{Proof of main theorems}\label{proof}

\begin{proof}[{\bf Proof of Theorem \ref{th.explicit-RS}}]
We note that $\chi_p(d_0)=\Bigl(\frac{d_0} p \Bigr)$ for any prime number $p$ and fundamental discriminant $d_0$. Hence by 
Theorem \ref{th.explicit-local-factor}, 
for any fundamental discriminant  $d_0$, we have
\begin{align*}
&\prod_p H(d_0,\iota,\alpha_{1,p},\alpha_{2,p},p^{-s+k_1/2+k_2/2})=|d_0|^{-s+k_1/2+k_2/2-1}L(n/2,\Bigl(\frac{d_0} * \Bigr))\prod_{i=1}^{(n-2)/2} \zeta(2i)\\
& \times \prod_p L_p(\Bigl(\frac{d_0} p \Bigr), \alpha_{1,p}+\alpha_{1,p}^{-1},\alpha_{2,p}+\alpha_{2,p}^{-1},p^{(n-1)/2}+p^{(1-n)/2},p^{-2s+k_1+k_2-n/2-3/2})\times L(s,f_1 \otimes f_2 \otimes G_n) \\
& \times \bigl(\zeta(2s+n-k_1-k_2+1)\prod_{i=1}^{n/2-1} \zeta(4s+2n-2k_1-2k_2+2-2i)\Bigr)^{-1}\times \prod_{i=1}^{n/2-1} L(2s-2i,f_1 \otimes f_2)
\end{align*}
Moreover, by (2) of Theorem \ref{th.explicit-local-factor},
\[\prod_p H(d_0,\epsilon,\alpha_{1,p},\alpha_{2,p},p^{-s+k_1/2+k_2/2}) \not=0\] only if $d_0=1$, and 
\begin{align*}
&\prod_p H(1,\epsilon,\alpha_{1,p},\alpha_{2,p},p^{-s+k_1/2+k_2/2})=(-1)^{n(n+2)/8}\zeta(n/2)\prod_{i=1}^{(n-2)/2} \zeta(2i) \\
&\phantom{xxxxxxxxxxxxxx} \times \bigl(\zeta(2s+n-k_1-k_2+1)\prod_{i=1}^{n/2} \zeta(4s+2n-2k_1-2k_2+2-2i)\Bigr)^{-1}\\
&\phantom{xxxxxxxxxxxxxx} \times \zeta(2s+n/2+1-k_1-k_2) \prod_{i=1}^{n/2} L(2s-2i+1,f_1 \otimes f_2).
\end{align*}
Note that $\kappa_n$ in (\ref{kappa}) can be written as $\kappa_n= 2^{1-n/2}\Gamma_{\CC}(n/2)\prod_{i=1}^{n/2-1} \Gamma_{\CC}(2i),$
and
\[L(n/2,\Bigl(\frac {d_0} * \Bigr))=\pm \Gamma_{\CC}(n/2)^{-1}|d_0|^{-n/2+1/2}L(1-n/2,\Bigl(\frac {d_0} *\Bigr))
\]
for any fundamental discriminant $d_0$.
We note that $2k_1-n, 2k_2-n$ and $n$ are the weight of $f_1,f_2$ and $G_n$, respectively. 
Thus by 
 Theorem \ref{th.local-global}, Theorem \ref{th.explicit-RS} follows.

\end{proof}

\begin{proof} [{\bf Proof of Theorem \ref{th.main-result}.}]  
(1) For an even positive integer $n$, put
\[\delta_n(s)=\delta_{n.k_1,k_2}(s)=
\begin{cases}
\displaystyle \prod_{i=1}^{n/2} \frac{\Gamma(s-(k_1+k_2)/2+(n-i+2)/2)}{\Gamma(s-(k_1+k_2)/2+i/2+1/2)}, & \text{if $n \equiv 0$ mod 4} \\
\displaystyle \prod_{i=1}^{n/2-1} \frac{\Gamma(s-(k_1+k_2)/2+(n-i+2)/2)}{\Gamma(s-(k_1+k_2)/2+i/2+1/2)}, & \text{if $n \equiv 2$ mod 4}. 
\end{cases}\]
Then $\delta_n(s)$ is a meromorphic function, and by the functional equation 
\[\Gamma(s)\Gamma(1-s)=\pi/\sin (\pi s),\] we see that it is invariant under the transformation 
$s \mapsto k_1+k_2-(n+1)/2-s.$
Put 
\begin{align*}
R_1(s,I_n(h_1),I_n(h_2))=\frac {\lambda_n2^{sn}}{\zeta(2s+n-k_1-k_2+1)} D(s;h_1,h_2,E_{n/2+1/2}) \prod_{i=1}^{\tfrac n2-1}
\frac {L(2s-2i,f_1 \otimes f_2)}{\zeta(4s+2n-2k_1-2k_2+2-2i)}
\end{align*}
\begin{align*}
& R_2(s,I_n(h_1),I_n(h_2))=\frac {\lambda_n2^{sn}}{\zeta(2s+n-k_1-k_2+1)}(-1)^{\tfrac {n(n-2)}8}\pi^{\tfrac n2}c_{h_1}(1){\overline{c_{h_2}(1)}}\zeta(1-\tfrac n2) \zeta(2s-k_1-k_2+\tfrac n2+1)\\
& \phantom{xxxxxxxxxxxxxxx}\times \prod_{i=1}^{\tfrac n2} \frac {L(2s-2i+1,f_1 \otimes f_2)}{\zeta(4s+2n-2k_1-2k_2+2-2i)}.
\end{align*}
Let, for $i=1,2$,
\begin{align*} &\calr_i(s,I_n(h_1),I_n(h_2)=\gamma_n(s)\xi(2s+n+1-k_1-k_2)\prod_{j=1}^{n/2}\xi(4s+2n+2-2k_1-2k_2-2j) R_i(s,I_n(h_1),I_n(h_2)).
\end{align*}

If $n \equiv 2 \text{ mod } 4$, then $\calr_2(s,I_n(h_1),I_n(h_2))=0$ since we have $c_{h_1}(1)=c_{h_2}(1)=0$.
Let $n \equiv 0 \text{ mod } 4$. Then we can show 
\begin{align*}
\calr_2(s,I_n(h_1),I_n(h_2))=c_2 \delta_n(s) \xi(2s-k_1-k_2+n/2+1) \prod_{i=1}^{n/2}\call(2s-2i+1,f_1 \otimes f_2),\end{align*}
with $c_2$ a constant.
By the holomorphy and functional equations of $\delta_n(s), \zeta(s)$ and $L(s,f_1 \otimes f_2)$, we see that
$\calr_2(s,I_n(h_1),I_n(h_2))$ is a meromorphic function of $s$ and satisfies
\[\calr_2(k_1+k_2-(n+1)/2-s,I_n(h_1),I_n(h_2))=\calr_2(s,I_n(h_1),I_n(h_2)).\]
Here we use the fact that $\prod_{i=1}^{n/2}\call(2s-2i+1,f_1 \otimes f_2)$ is invariant under the transformation 
$s \mapsto k_1+k_2-(n+1)/2-s.$

Now we can show that 
\begin{align*}
\calr_1(s,I_n(h_1),I_n(h_2))=c_1\delta_n(s) \cald(s;h_1,h_2,E_{n/2+1/2})\prod_{i=1}^{n/2-1} \call(2s-2i,f_1 \otimes f_2),
\end{align*}
with $c_2$ a constant, and $\calr_1(s,I_n(h_1),I_n(h_2))$ can be continued meromorphically to the whole $s$-plane, and 
\[\calr_1(k_1+k_2-(n+1)/2-s,I_n(h_1),I_n(h_2))=\calr_1(s,I_n(h_1),I_n(h_2)).\]
Then, the assertion (1) follows from the fact that $\prod_{i=1}^{n/2-1}\call(2s-2i,f_1 \otimes f_2)$ is invariant under the transformation 
$s \mapsto k_1+k_2-(n+1)/2-s.$

(2) Let $k_1=k_2=k$ and $h_1=h_2=h$. We note that $R_2(I_n(h),I_n(h),s)$, and $\prod_{i=1}^{n/2-1} L(2s-2i, f \otimes f)$ are  finite at $s=k$. Hence 
\begin{align*}
&\mathrm{Res}_{s=k} R(s,I_n(h_1),I_n(h_2)=b_n 2^{kn}\mathrm{Res}_{s=k} D(s,h,h,E_{n/2+1/2}) \prod_{i=1}^{n/2-1} L(2k-2i, f \otimes f),\end{align*}
with $b_n$ a non-zero constant. We note that $\prod_{i=1}^{n/2-1} L(2k-2i, f \otimes f)$ is non-zero, Hence, by (2) of Proposition \ref{prop.fc-RS} and Corollary \ref{th.Ikeda-conjecture},
\begin{align*}
\mathrm{Res}_{s=k} D(s,h,h,E_{n/2+1/2})=c_n 2^{-3kn+4k}\pi^{-kn+k}\frac{\Gamma(k)L(k,f) \prod_{i-1}^{n/2-1}\Gamma(2k-n+2i)}{ 2^{1-2kn} \pi^{-kn+n(n-1)/4}\prod_{i=1}^n \Gamma(k+\frac{1}{2}(-i+1))} \langle h,h\rangle,\end{align*}
where $c_n$ is a non-zero constant. 
Thus the assertion (2) is proved.
\end{proof}

\section{Mass Equidistribution}

Let $f_k$ be a holomorphic Hecke eigenform of weight $k$ with respect to $SL_2(\Bbb Z)$. Then the arithmetic quantum unique ergodicity (AQUE) proved by Holowinsky and Soundararajan \cite{HS} says that 
as $k\to\infty$,
$$\frac {|f_k(Z)|^2}{\langle f_k,f_k\rangle} y^k \frac {dxdy}{y^2}\longrightarrow \frac 3{\pi^2} \frac {dxdy}{y^2}.
$$

Cogdell and Luo \cite{CL} considered a generalization of AQUE for Siegel modular forms. Namely,
let $F_k$ be a holomorphic Siegel cusp form of weight $k$ with respect to $\Gamma=Sp_{n}(\Bbb Z)$. Then it is expected that as $k\to\infty$,
$$\frac {|F_k(Z)|^2}{\langle F_k,F_k\rangle} det(Y)^k \frac {dXdY}{det(Y)^{n+1}}\longrightarrow \frac 1{vol(\Gamma\backslash \Bbb H_n)} \frac {dXdY}{det(Y)^{n+1}}.
$$

This means that for any $\Phi$ in $L^2(\Gamma\backslash \Bbb H_n)$, as $k\to\infty$,
$$
\int_{\Gamma\backslash \Bbb H_n} \Phi(Z) \frac {|F_k(Z)|^2}{\langle F_k,F_k\rangle} det(Y)^k \frac {dXdY}{det(Y)^{n+1}}\longrightarrow \frac 1{vol(\Gamma\backslash \Bbb H_n)}\int_{\Gamma\backslash \Bbb H_n} \Phi(Z)\frac {dXdY}{det(Y)^{n+1}}.
$$

Liu \cite{Liu} verified it in the case when $F_k$ is the Ikeda lift and $\Phi$ is the Klingen Eisenstein series. In this section, we show it when $F_k$ is the D-I-I lift and $\Phi$ is the Siegel Eisenstein series under the assumption of the holomorphy of $D(s;h,h,E_{n/2+1/2})$ for $h\in S_{k-n/2+1/2}^+(SL_2(\Bbb Z))$. Namely,
\begin{equation}\label{assump}
\text{$\mathcal D(s,h,h,E_{n/2+1/2})$ is holomorphic except possibly at $k-\frac j4$ for $j=0,1,...,2n+2$}.
\end{equation}

When $\Phi(Z)=E_{n,0}(Z,\tfrac {n+1}4+it)$ (the center of the critical strip), then 
$$\int_{\Gamma\backslash \Bbb H_n} E_{n,0}(Z,\tfrac {n+1}4+it)\frac {dXdY}{det(Y)^{n+1}}=0.
$$
This is a well-known result, and it can be proved adelically as follows: The Siegel Eisenstein series is an iterated residue of the Borel Eisenstein series $E(g,\varphi,\lambda)$ in the notation of \cite[Corollary 17]{JLR}. Let $G=Sp_n$. Then by \cite[Corollary 2]{KW}, $\int_{G(\Bbb Q)\backslash G(\Bbb A)} \wedge^T E(g,\varphi,\lambda)\, dg=\int_{\mathfrak F(T)} E(g,\varphi,\lambda)\, dg$, where $\wedge^T$ is the truncation operator, and $\mathfrak F(T)$ is the truncated fundamental domain. By the formula in \cite[Corollary 17]{JLR}, the LHS$\to 0$ as $T=x_1e_1+\cdots +x_n(e_1+\cdots+e_n)$ and $x_i\to\infty$ if $Re(\langle \rho-w\lambda, e_1+\cdots+e_i\rangle) >0$ for each $i=1,...,n$. It is the case in our situation.

So we expect, as $k\to\infty$,
$$
\int_{\Gamma\backslash \Bbb H_n} \Phi(Z) \frac {|F_k(Z)|^2}{\langle F_k,F_k\rangle} det(Y)^k \frac {dXdY}{det(Y)^{n+1}}\longrightarrow 0.
$$
We prove a more precise decay when $F_k=I_n(h)$: 

\begin{theorem}\label{main-Siegel}
Let $E_{n,0}(Z,s)$ be the Siegel-Eisenstein series. For $h\in S_{k-\tfrac n2+\tfrac 12}^+(\Gamma_0(4))$, let $I_n(h)$ be the D-I-I lift. 
Then under (\ref{assump}) and $n\geq 4$,
$$
\int_{\Gamma\backslash\Bbb H_n} \frac {|I_n(h)(Z)|^2}{\langle I_n(h),I_n(h)\rangle} E_{n,0}(Z,\tfrac {2n+1}4+it) \det(Y)^{k}\,d^*Z
\ll_{t,n} k^{-\tfrac {n^2+2n-8}8-\epsilon}.
$$
\end{theorem}

\subsection{Convexity bound}

For two Siegel cusp forms $F,G$ of weight $k$, let $R(s,F,G)$ be the Rankin-Selberg convolution. Then from Proposition \ref{prop.fc-RS}, 

\begin{equation}\label{residue-R}
\mathrm{Res}_{s=k} R(s,F,F)=\frac {\langle F,F\rangle}{\gamma(k)\xi(n+1)} \prod_{j=1}^{[\frac n2]} \frac {\xi(2j+1)}{\xi(2n+2-2j)}.
\end{equation}

The critical line is $\mathrm{Re}(s)=k-\frac {n+1}4$. 
By Ikehara Tauberian theorem,
$\sum_{det(T)\leq X} \frac {|a_F(T)|^2}{\epsilon(T)}\sim X \mathrm{Res}_{s=k} R(s,F,F)$.
So by partial summation, we have $R(k+\epsilon,F,F)\ll Res_{s=k} R(s,F,F)$.
Then by using the functional equation, we have

$$R(k-\tfrac {n+1}2-\epsilon+it,F,F)\ll_{t,n} |R(k+\epsilon,F,F)| k^{\frac {n(n+1)}2+\epsilon}.
$$
Hence by Phragmen-Lindel\"of principle, we have the convex bound:

$$R(k-\tfrac {n+1}4+it,F,F)\ll_{t,n} |R(k+\epsilon,F,F)| k^{\frac {n(n+1)}4+\epsilon}.
$$

We need a subconvexity bound of the form:

\begin{conjecture}\label{sub-sp} There exists $\delta>0$ such that 
$$R(k-\tfrac {n+1}4+it,F,F)\ll_{t,n} (\mathrm{Res}_{s=k} R(s,F,F)) k^{\frac {n(n+1)}4-\delta}.
$$
\end{conjecture}

Under Conjecture \ref{sub-sp}, we have
$$
\int_{\Gamma\backslash \Bbb H_n} \frac {|F(Z)|^2}{\langle F,F\rangle} E(Z,\tfrac {n+1}4+it) \det(Y)^{k}\, \frac {dXdY}{det(Y)^{n+1}}
\ll_{t,n} \frac {\gamma(k-\frac {n+1}4+it)}{\gamma(k)} k^{\frac {n(n+1)}4-\delta}
\ll_{t,n} k^{-\delta}.
$$

\subsection{Proof of Conjecture \ref{sub-sp} for the D-I-I lift under (\ref{assump})}

For $h\in S_{k-n/2+1/2}^+(SL_2(\Bbb Z))$, let $I_n(h)$ be the D-I-I lift. For simplicity, we denote $D(s;h,h,E_{n/2+1/2})$ by $D(s,h)$.
From (\ref{th.explicit-RS}), we have
\begin{equation}\label{D-R}
\mathrm{Res}_{s=k} D(s,h)=c_n \frac {\langle F_f,F_f\rangle}{\gamma(k)2^{kn} \prod_{i=1}^{\tfrac n2-1} L(2k-2i,f\otimes f)}
=c_n' \frac {\mathrm{Res}_{s=k} R(s,I_n(h),I_n(h))}{2^{kn} \prod_{i=1}^{\tfrac n2-1} L(2k-2i,f\otimes f)},
\end{equation}
for some constants $c_n,c_n'$.

By Ikehara Tauberian theorem and partial summation, we have 
$$D(k+\epsilon+it,h)\ll_{t,n} \mathrm{Res}_{s=k} D(s,h).
$$

Then under the assumption (\ref{assump}), and the functional equation, we have
$$D(k-\tfrac n2-\tfrac 12-\epsilon+it,h)\ll_{t,n} D(k+\epsilon-it,h) k^{n+1+\epsilon}.
$$

By Phragmen-Lindel\"of principle,

\begin{equation}\label{P-L}
D(k-\tfrac n4-\tfrac 14-\epsilon+it,h)\ll_{t,n} D(k+\epsilon-it,h)k^{\tfrac {n+1}2+\epsilon}\ll k^{\tfrac {n+1}2+\epsilon} \mathrm{Res}_{s=k} D(s,h).
\end{equation}

We have
$$L(2k-n+\epsilon,f\otimes f)\ll \mathrm{Res}_{s=2k-n} L(s,f\otimes f)\ll L(1,Sym^2\pi_f)\ll k^\epsilon.
$$
From the functional equation (\ref{f-f-RS}) of $L(s,f\otimes f)$ and by Phragmen-Lindel\"of principle, if $s=2k-n-j+it$ with $j\geq 1$,
\begin{eqnarray}\label{P-L1}
&& |L(2k-n-j+it,f\otimes f)|\ll_{t,n} \left|\frac {\Gamma(2k-n+j-1-it)}{\Gamma(2k-n-j+it)}\right| |L(2k-n+j-1-it,f\otimes f)|\\
&& \phantom{xxxxxxxxxxxxxxxxxxxx} \ll_{t,n} k^{2j-1}|L(2k-n+j-1-it,f\otimes f)|\ll k^{2j-1}.\nonumber
\end{eqnarray}

By convexity bound, $L(2k-n-\tfrac 12+it,f\otimes f)\ll_{t,n} k^{\tfrac 12+\epsilon}$.

Now we compute $R(s,I_n(h),I_n(h))$ at $s=k-\tfrac n4-\tfrac 14+it$. 
We divide into two cases:

Case 1. $n=4l+2$. 
In this case, the second sum in (\ref{explicit-formula}) is zero since $c_h(1)=0$.

Then for $s=k-l-\tfrac 34+it$,
\begin{eqnarray*}
&& R(k-l-\tfrac 34+it,I_n(h),I_n(h))\ll_{t,n} 2^{k(4l+2)} |D(k-l-\tfrac 34+it,h)| \prod_{j=1}^{2l} |L(2k-2l-\tfrac 32-2j+2it,f\otimes f)|.
\end{eqnarray*}

From (\ref{D-R}), (\ref{P-L}) and (\ref{P-L1}), in the product, $j=1,...,l$ contributes to $O(k^\epsilon)$. Hence
$$|\prod_{j=1}^{2l} |L(2k-2l-\tfrac 32-2j+2it,f\otimes f)|\leq k^{\sum_{i=1}^l 2(2i-1)}=k^{2l^2+\epsilon}.
$$
Also $L(2k-2i,f\otimes f)\gg 1$ for each $i=1,...,\tfrac n2-1$.
Therefore, 
$$R(k-l-\tfrac 34+it,I_n(h),I_n(h))\ll_{t,n} \mathrm{Res}_{s=k} R(s,I_n(h),I_n(h)) k^{\frac {n^2}8+1+\epsilon}.
$$

This verifies Conjecture \ref{sub-sp} except for $n=2$.

Case 2. $n=4l$. In the same way as above, we can estimate the term in the first sum: 
for $s=k-\frac {n+1}4+it$, in the product 
$$\prod_{j=1}^{\tfrac n2-1} L(2k-\tfrac n2-\tfrac 12-2j+2it,f\otimes f)=\prod_{j=1}^{2l} L(2k-2l-\tfrac 12-2j+2it, f\otimes f),
$$
$j=1,...,l-1$ contribute to $O(k^{\epsilon})$. When $j=l$, it gives rise to the central value. So by (\ref{P-L1}),
$$\prod_{j=1}^{2l} L(2k-2l-\tfrac 12-2j+2it, f\otimes f)
\ll_{t,n} k^{\sum_{i=1}^{l-1} 4i} k^{\tfrac 12+\epsilon}\ll_{t,n} k^{2l(l-1)+\tfrac 12+\epsilon}=k^{\tfrac {n(n-4)}8+\tfrac 12+\epsilon}.
$$
Therefore, the first sum is majorized by
$$
\ll_{t,n} 2^{k(4l)} |D(k-l-\tfrac 14+it,h)| \prod_{j=1}^{2l-1} |L(2k-2l-\tfrac 12-2j+2it,f\otimes f)|
\ll \mathrm{Res}_{s=k} R(s,I_n(h),I_n(h)) k^{\frac {n^2}8+1+\epsilon}.
$$

For the second sum, recall the following identity \cite[Theorem 1]{K-Z}:
\begin{equation*}
\frac {|c_{h}(1)|^2}{\langle h,h\rangle}=\frac {\Gamma(k-\tfrac n2)}{\pi^{k-\tfrac n2}} \frac {L(k-\tfrac n2,f)}{\langle f,f\rangle}.
\end{equation*}
Then by Theorem \ref{th.Ikeda-conjecture}, and the fact that $\langle f,f\rangle=\frac {\pi^{-2k+n}}{12} \Gamma(2k-n)L(1,f,{\rm Ad})$,

$$|c_h(1)|^2=\frac {\Gamma(k-\tfrac n2)L(k-\tfrac n2,f)}{\pi^{k-\tfrac n2}\langle f,f\rangle}\langle h,h\rangle
=e_n \frac {2^{-kn-4k}L(k-\tfrac n2,f) \mathrm{Res}_{s=k} R(s,I_n(h),I_n(h))}{L(k,f)L(1,f,{\rm Ad}) \prod_{i=1}^{\tfrac n2-1} L(2k-2i,f\otimes f)},
$$
for some constant $e_n$. By convexity bound, $L(k-\tfrac n2,f)\ll_n k^\frac 12$, and $L(1,f,{\rm Ad})\gg_n k^{-\epsilon}$. 
Moreover, by Deligne's estimate, we have
$L(k,f) \ge \frac{\zeta(n+1)^2}{\zeta(n/2+1/2)^2}$.
When $s=k-l-\frac 14+it$, the second sum is
\begin{equation}\label{second}
\ll_{n,t} \mathrm{Res}_{s=k} R(s,I_n(h),I_n(h)) \frac {2^{-kn-4k} L(k-\tfrac n2,f)}{L(k,f)L(1,f,{\rm Ad})} 
\prod_{j=1}^{2l} \left|L(2k-2l-\tfrac 12-2j+2it,f\otimes f)\right|.
\end{equation}

Here 
$$\prod_{j=1}^{2l} \left|L(2k-2l-\tfrac 12-2j+2it,f\otimes f)\right|\ll_{t,n} k^{2l(l+1)+\tfrac 12+\epsilon}.
$$

Therefore, (\ref{second}) has exponential decay as $k\to\infty$. Therefore,
 
$$R(k-l-\tfrac 14+it,I_n(h),I_n(h))\ll_{t,n} \mathrm{Res}_{s=k} R(s,I_n(h),I_n(h)) k^{\frac {n^2}8+1+\epsilon}.
$$

This verifies Conjecture \ref{sub-sp}.

\bigskip

\end{document}